\documentclass[a4paper,twoside]{article}
\usepackage{a4}
\usepackage{amssymb}
\usepackage{amsmath}
\usepackage{url}
\usepackage[active]{srcltx}
\usepackage[dvipsnames]{xcolor} % Should be loaded before pagebackref
\usepackage[pagebackref,colorlinks,citecolor=blue,linkcolor=blue,urlcolor=blue]{hyperref}
\usepackage{amsthm}
\usepackage{upref}
\newcount\minutes \newcount\hours
\hours=\time
\divide\hours 60
\minutes=\hours
\multiply\minutes -60
\advance\minutes \time
\newcommand{\klockan}{\the\hours:{\ifnum\minutes<10 0\fi}\the\minutes}
\newcommand{\tid}{\today\ \klockan}
\newcommand{\prtid}{\smash{\raise 10mm \hbox{\LaTeX ed \tid}}}
\renewcommand{\prtid}{}
\makeatletter
\def\sectionmark#1{} %\markboth{{\sectnr #1}}{{\sectnr #1}}} %Journal
\def\subsectionmark#1{}
\newcommand{\sectnr}{\ifnum \c@secnumdepth >\z@
   \thesection.\hskip 1em\relax \fi}
\def\@evenhead{\footnotesize\rm\thepage\hfil\leftmark\hfil\llap{\prtid}}
\def\@oddhead{\footnotesize\rm\rlap{\prtid}\hfil\rightmark\hfil\thepage}
\def\tableofcontents{\section*{Contents} %\@mkboth{Contents}{Contents}} %Journal
 \@starttoc{toc}}
\makeatother
\makeatletter
\def\@biblabel#1{#1.}
\makeatother
\makeatletter
\let\Thebibliography=\thebibliography
\renewcommand{\thebibliography}[1]{\def\@mkboth##1##2{}\Thebibliography{#1}
\addcontentsline{toc}{section}{References}
\frenchspacing % Maybe not needed
\setlength{\@topsep}{0pt}% Delete if extra space before list
\setlength{\itemsep}{0pt}%
\setlength{\parskip}{0pt plus 2pt}%
}
\makeatother
\makeatletter
\def\mdots@{\mathinner.\nonscript\!.%
 \ifx\next,.\else\ifx\next;.\else\ifx\next..\else
 \nonscript\!\mathinner.\fi\fi\fi}
\let\ldots\mdots@
\let\cdots\mdots@
\let\dotso\mdots@
\let\dotsb\mdots@
\let\dotsm\mdots@
\let\dotsc\mdots@
\def\vdots{\vbox{\baselineskip2.8\p@ \lineskiplimit\z@
    \kern6\p@\hbox{.}\hbox{.}\hbox{.}\kern3\p@}}
\def\ddots{\mathinner{\mkern1mu\raise8.6\p@\vbox{\kern7\p@\hbox{.}}%
    \raise5.8\p@\hbox{.}\raise3\p@\hbox{.}\mkern1mu}}
\makeatother
\makeatletter
\let\Enumerate=\enumerate
\renewcommand{\enumerate}{\Enumerate%
\setlength{\itemsep}{0pt}%
\setlength{\parskip}{0pt plus 1pt}%
\renewcommand{\theenumi}{\textup{(\alph{enumi})}}%
\renewcommand{\labelenumi}{\theenumi}%
}

\makeatother
\makeatletter
\def\@seccntformat#1{\csname the#1\endcsname.\quad}
\makeatother
\newcommand{\authortitle}[2]{\author{#1}\title{#2}\markboth{#1}{#2}}
\newcommand{\auth}[2]{{#1, #2.}}
\newcommand{\art}[6]{{\sc #1, \rm #2, \it #3 \bf #4 \rm (#5), \mbox{#6}.}}
\newcommand{\artin}[3]{{\sc #1, \rm #2,  in #3.}}
\newcommand{\artprep}[3]{{\sc #1, \rm #2, #3.}}

\newcommand{\book}[3]{{\sc #1, \it #2, \rm #3.}}
\newcommand{\AND}{{\rm and }}
\newcommand{\artnopt}[6]{{\sc #1, \rm #2, \it #3\/ \bf #4 \rm (#5), \mbox{#6}}}
\newcommand{\arXiv}[1]{{\tt \href{https://arxiv.org/abs/#1}{arXiv:#1}}}
\RequirePackage{amsthm}
\newtheoremstyle{descriptive}%
  {\topsep}   %{\medskipamount}          % Space above
  {\topsep}   %  {\medskipamount}          % Space below
  {\rmfamily} % Body font
  {}          % Indent
  {\bfseries} % Head font
  {.}         % Punctuation after thm head
  { }         % Space after thm head
  {}          % Thm head spec(?)
\newtheoremstyle{propositional}%
  {\topsep}   %  {\medskipamount}          % Space above
  {\topsep}   %  {\medskipamount}          % Space below
  {\itshape}  % Body font
  {}          % Indent
  {\bfseries} % Head font
  {.}         % Punctuation after thm head
  { }         % Space after thm head
  {}          % Thm head spec(?)
\newtheoremstyle{remarkstyle}%
  {\topsep}   %  {\medskipamount}          % Space above
  {\topsep}   %  {\medskipamount}          % Space below
  {\rmfamily}  % Body font
  {}          % Indent
  {\itshape} % Head font
  {.}         % Punctuation after thm head
  { }         % Space after thm head
  {}          % Thm head spec(?)
\theoremstyle{propositional}
\newtheorem{thm}{Theorem}[section]
\newtheorem{prop}[thm]{Proposition}
\newtheorem{lem}[thm]{Lemma}
\newtheorem{cor}[thm]{Corollary}
\theoremstyle{descriptive}
\newtheorem{deff}[thm]{Definition}
\newtheorem{example}[thm]{Example}
\newtheorem{remark}[thm]{Remark}
\makeatletter
\renewenvironment{proof}[1][\proofname]{\par
  \pushQED{\qed}%
  \normalfont
  \trivlist
  \item[\hskip\labelsep
        \itshape
    #1\@addpunct{.}]\ignorespaces
}{%
  \popQED\endtrivlist\@endpefalse
}
\makeatother
\newcommand{\setm}{\setminus}
\renewcommand{\subsetneq}{\varsubsetneq}
\renewcommand{\emptyset}{\varnothing}
{\catcode`p =12 \catcode`t =12 \gdef\eeaa#1pt{#1}}      % Get slantfactor
\def\accentadjtext#1{\setbox0\hbox{$#1$}\kern   % Convert it with height
                \expandafter\eeaa\the\fontdimen1\textfont1 \ht0 }
\def\accentadjscript#1{\setbox0\hbox{$#1$}\kern % Convert it with height
                \expandafter\eeaa\the\fontdimen1\scriptfont1 \ht0 }
\def\accentadjscriptscript#1{\setbox0\hbox{$#1$}\kern   % Convert it with height
                \expandafter\eeaa\the\fontdimen1\scriptscriptfont1 \ht0 }
\def\accentadjtextback#1{\setbox0\hbox{$#1$}\kern       % Convert it with height
                -\expandafter\eeaa\the\fontdimen1\textfont1 \ht0 }
\def\accentadjscriptback#1{\setbox0\hbox{$#1$}\kern     % Convert it with height
                -\expandafter\eeaa\the\fontdimen1\scriptfont1 \ht0 }
\def\accentadjscriptscriptback#1{\setbox0\hbox{$#1$}\kern % Convert it with height
                -\expandafter\eeaa\the\fontdimen1\scriptscriptfont1 \ht0 }
\def\itoverline#1{{\mathsurround0pt\mathchoice
        {\rlap{$\accentadjtext{\displaystyle #1}
                \accentadjtext{\vrule height1.593pt}
                \overline{\phantom{\displaystyle #1}
                \accentadjtextback{\displaystyle #1}}$}{#1}}
        {\rlap{$\accentadjtext{\textstyle #1}
                \accentadjtext{\vrule height1.593pt}
                \overline{\phantom{\textstyle #1}
                \accentadjtextback{\textstyle #1}}$}{#1}}
        {\rlap{$\accentadjscript{\scriptstyle #1}
                \accentadjscript{\vrule height1.593pt}
                \overline{\phantom{\scriptstyle #1}
                \accentadjscriptback{\scriptstyle #1}}$}{#1}}
        {\rlap{$\accentadjscriptscript{\scriptscriptstyle #1}
                \accentadjscriptscript{\vrule height1.593pt}
                \overline{\phantom{\scriptscriptstyle #1}
                \accentadjscriptscriptback{\scriptscriptstyle #1}}$}{#1}}}}
\def\itunderline#1{{\mathsurround0pt\mathchoice
        {\rlap{$\underline{\phantom{\displaystyle #1}
                \accentadjtextback{\displaystyle #1}}$}{#1}}
        {\rlap{$\underline{\phantom{\textstyle #1}
                \accentadjtextback{\textstyle #1}}$}{#1}}
        {\rlap{$\underline{\phantom{\scriptstyle #1}
                \accentadjscriptback{\scriptstyle #1}}$}{#1}}
        {\rlap{$\underline{\phantom{\scriptscriptstyle #1}
                \accentadjscriptscriptback{\scriptscriptstyle #1}}$}{#1}}}}
\def\cprime{{\mathsurround0pt$'$}}
\newcommand{\limplus}{{\mathchoice{\vcenter{\hbox{$\scriptstyle +$}}}
  {\vcenter{\hbox{$\scriptstyle +$}}}
  {\vcenter{\hbox{$\scriptscriptstyle +$}}}
  {\vcenter{\hbox{$\scriptscriptstyle +$}}}
}}
\newcommand{\limminus}{{\mathchoice{\vcenter{\hbox{$\scriptstyle -$}}}
  {\vcenter{\hbox{$\scriptstyle -$}}}
  {\vcenter{\hbox{$\scriptscriptstyle -$}}}
  {\vcenter{\hbox{$\scriptscriptstyle -$}}}
}}
\newcommand{\Cpt}{{C_{s,p}}}
\newcommand{\Csp}{{C_{s,p}}}
\newcommand{\Cop}{{C_{1,p}}}
\newcommand{\cpt}{{\capp_{s,p}}}
\newcommand{\csp}{{\capp_{s,p}}}
\DeclareMathOperator{\diam}{diam}
\DeclareMathOperator{\dist}{dist}
\DeclareMathOperator{\capp}{cap}
\DeclareMathOperator{\Lip}{Lip}
\newcommand{\Lipc}{{\Lip_c}}
\DeclareMathOperator*{\esssup}{ess\,sup}
\DeclareMathOperator{\Tail}{Tail}
\newcommand{\bdry}{\partial}
\newcommand{\bdy}{\bdry}
\newcommand{\simle}{\lesssim}
\newcommand{\loc}{_{\rm loc}}
\newcommand{\ga}{\gamma}
\newcommand{\eps}{\varepsilon}
\newcommand{\Om}{\Omega}
\newcommand{\gt}{\tilde{g}}
\renewcommand{\phi}{\varphi}
\newcommand{\p}{{$p\mspace{1mu}$}}
\newcommand{\R}{\mathbf{R}}
\newcommand{\Q}{\mathbf{Q}}
\newcommand{\Rn}{{\R^n}}
\newcommand{\eR}{{\overline{\R}}}
\def\cprime{{\mathsurround0pt$'$}}
\newcommand{\clOm}{\overline{\Om}}
\newcommand{\LL}{\mathcal{L}}
\newcommand{\UU}{\mathcal{U}}
\newcommand{\uP}{\itoverline{P}}     
\newcommand{\lP}{\itunderline{P}} 
\DeclareMathOperator{\supp}{supp}
\newcommand{\E}{\mathcal{E}}
\newcommand{\clG}{\itoverline{G}}
\newcommand{\Vsp}{V^{s,p}}
\newcommand{\VspOm}{V^{s,p}(\Om)}
\newcommand{\Vst}{V^{s,2}}
\newcommand{\VstOm}{V^{s,2}(\Om)}
\newcommand{\Vspo}{V^{s,p}_0}
\newcommand{\VspoOm}{V^{s,p}_0(\Om)}
\newcommand{\Wsp}{W^{s,p}}
\newcommand{\Wsploc}{W^{s,p}\loc}
\newcommand{\Wp}{W^{1,p}}
\newcommand{\La}{\Lambda}
\newcommand{\Omc}{{\Om^c}}
\newcommand{\Gc}{G^c}
\newcommand{\clB}{\itoverline{B}}

\numberwithin{equation}{section}
\newcommand{\eqv}{\ensuremath{
\mathchoice{\quad \Longleftrightarrow \quad}{\Leftrightarrow}
                {\Leftrightarrow}{\Leftrightarrow}}}
\newcommand{\imp}{\ensuremath{
\mathchoice{\quad \Longrightarrow \quad}{\Rightarrow}
                {\Rightarrow}{\Rightarrow}}}
\newenvironment{ack}{\medskip{\it Acknowledgement.}}{}

\begin{document}

\authortitle{Anders Bj\"orn, Jana Bj\"orn and Minhyun Kim}
            {Semiregular and strongly irregular boundary points 
for nonlocal Dirichlet problems}

\author{
Anders Bj\"orn \\
\it\small Department of Mathematics, Link\"oping University, SE-581 83 Link\"oping, Sweden\\
\it \small anders.bjorn@liu.se, ORCID\/\textup{:} 0000-0002-9677-8321
\\
\\
Jana Bj\"orn \\
\it\small Department of Mathematics, Link\"oping University, SE-581 83 Link\"oping, Sweden\\
\it \small jana.bjorn@liu.se, ORCID\/\textup{:} 0000-0002-1238-6751
\\
\\
 Minhyun Kim \\
 \it\small Department of Mathematics \& Research Institute for Natural Sciences, \\
 \it\small Hanyang University, 04763 Seoul, Republic of Korea \\
 \it \small minhyun@hanyang.ac.kr, ORCID\/\textup{:} 0000-0003-3679-1775
}

\date{Preliminary version, \today}
\date{}

\maketitle

\noindent{\small
{\bf Abstract}. 
In this paper we study nonlocal nonlinear equations of 
fractional $(s,p)$-Laplacian type on $\mathbf{R}^n$.
We show that the irregular boundary points for
the Dirichlet problem can be divided into two
disjoint classes: semiregular and strongly irregular boundary points,
with very different behaviour. 
Two fundamental tools needed to show this are the Kellogg property (from
our previous paper) and a new removability result for 
solutions in the $V^{s,p}$ Sobolev type space,
which we deduce more generally also for 
supersolutions of
equations with a right-hand side.
Semiregular and strongly irregular points are also characterized in various ways.
Finally, it is explained how semiregularity depends on $s$ and $p$.
}

\medskip


\medskip


%For submissions we can add declarations like this.
%I would skip that for arXiv.

%% \medskip
%% \noindent
%%     {\small {\bf Declarations}. \\
%%       \emph{Funding}:
%% A.~B. and J.~B. were supported by the Swedish Research Council,
%%   grants 2020-04011 and 2022-04048, respectively.
%% MK was supported 
%% by the National Research Foundation of Korea (NRF) grant
%% funded by the Korean government (MSIT) (RS-2023-00252297).
%%      \\
%% \emph{Conflicts of interest}:   None.
%% \\
%% \emph{Availability of data and material}:   Not applicable.
%% \\
%% \emph{Code availability}:   Not applicable.
%%     } 

\section{Introduction}

The aim of this paper is to study 
irregular boundary points 
for the nonlocal nonlinear equation $\LL u =0$.
The operator $\LL$ is of the form
\begin{equation*}
\mathcal{L}u(x) = 2 \,\mathrm{p.v.} \int_{\R^n} |u(x)-u(y)|^{p-2} (u(x)-u(y)) k(x, y) \,dy,
\end{equation*}
where $1<p<\infty$, $n \ge 1$ and $k: \R^n \times \R^n \to [0, \infty]$ is a symmetric measurable kernel
that satisfies the ellipticity condition
\begin{equation} \label{eq-comp-(x,y)}
\frac{\Lambda^{-1}}{|x-y|^{n+sp}} \leq k(x, y) \leq \frac{\Lambda}{|x-y|^{n+sp}},
\end{equation}
with $0<s<1$ and $\La \ge1$.
Note that $\LL$ is the fractional \p-Laplacian $(-\Delta_p)^s$ when 
$k(x,y)=|x-y|^{-n-sp}$.

Nonlocal and fractional problems have been studied for a long time,
see e.g.\ \cite{AH,BH86}, but in recent years, starting perhaps
with the pioneering paper by Caffarelli--Silvestre~\cite{CafSil},
they have received
a lot more attention (see e.g.\ our list of references).
Fractional operators also have significant applications in various contexts,
such as the obstacle problem \cite{CF13,Sil07}, 
material science \cite{Bat06,BC99}, 
phase transitions \cite{CSM05,FV11,SV09}, 
soft thin films \cite{Kur06}, quasi-geostrophic dynamics \cite{CV10} 
and image processing \cite{GO08}.

We assume in the rest of the introduction,
except for Theorem~\ref{thm-remove},
that  $\Om \subset \Rn$ is a nonempty bounded open set.
In the \emph{Dirichlet problem} one seeks 
a solution of $\LL u=0$
in $\Om$ which takes on some
prescribed exterior Dirichlet data 
$g: \Omc \to \eR:=[-\infty,\infty]$. 
Due to the nonlocal nature of the problem the Dirichlet data
have to be prescribed on the whole complement $\Omc$ of $\Om$.
For this reason, the Dirichlet problem is 
also called the exterior value problem or the complement value problem, 
but we will call it the Dirichlet problem for short.

As  in the classical setting of harmonic functions one can in general
not solve the Dirichlet problem so that the prescribed boundary values 
are taken as limits on $\bdy \Om$,
even if $g$ is bounded and continuous.
Hence, the Dirichlet problem has to be understood in a relaxed sense.
In this paper we 
use Sobolev and Perron solutions of the Dirichlet problem.
We
consider Sobolev solutions $Hg$
for Sobolev data $g$ in the fractional Sobolev type space $\VspOm$
naturally associated with the operator $\LL$,
see Theorem~\ref{thm-DP} for the precise definition.
For the nonlinear nonlocal problems $\LL u =0$ as above,
the Dirichlet problem was solved in the Sobolev sense 
by Di Castro--Kuusi--Palatucci~\cite{DCKP16}.

A boundary point $x_0 \in \bdy \Om$ is \emph{regular} 
  if 
  \[
    \lim_{\Om \ni x \to x_0} H g(x)=g(x_0)
    \quad \text{for every } g \in \VspOm \cap C(\R^n).
  \]  
Kim--Lee--Lee~\cite{KLL23} obtained the Wiener criterion for regular
boundary points.
Several
other characterizations of regular boundary points
were given in our earlier paper~\cite{BBK1}, 
see also Lindgren--Lindqvist~\cite{LL17}.

Boundary regularity can be paraphrased in the following way:
$x_0$ is 
regular if the following two conditions hold:
\begin{enumerate}
\renewcommand{\theenumi}{\textup{(\Roman{enumi})}}%
\renewcommand{\labelenumi}{\theenumi}%
\item \label{aa}
For every $g \in \VspOm \cap C(\R^n)$, the limit 
$    \lim_{\Om \ni x \to x_0} H g(x)$ exists.
\item \label{ab}
For every $g \in \VspOm \cap C(\R^n)$, there is a sequence 
$\{y_j\}_{j=1}^\infty$ such that
\begin{equation} \label{eq-strong-irr}
\Om \ni y_j \to x_0 
\quad \text{and} \quad H g(y_j) \to g(x_0),
\qquad \text{as } j \to \infty.
\end{equation}
\end{enumerate}

An irregular (i.e.\ nonregular) boundary point $x_0 \in \bdy \Om$
thus fails at least one of the two conditions~\ref{aa} and~\ref{ab}.
We say that an irregular point is \emph{semiregular} if 
\ref{aa} holds and that it is
\emph{strongly irregular} if  \ref{ab} holds.
Perhaps surprisingly it can  never happen that both
\ref{aa} and \ref{ab} fail.
This fact is part of the following main theorem in this paper.
As far as we know, our results are new also for the classical linear case $p=2$.

\begin{thm} \label{thm-trich}
\textup{(Trichotomy)}
  Let $x_0 \in \bdy \Om$.
  Then $x_0$ is either regular, semiregular or strongly irregular.
\end{thm}

As a direct consequence of the proof we obtain the following
result.
We refer the reader to Section~\ref{sect-cap}
for the definition of the Sobolev capacity $\Csp$.

\begin{prop} \label{prop-S-largest}
The set $S$ of semiregular points  is the largest relatively open subset
of $\bdy \Om \setm \bdy \clOm$ with zero $\Csp$ capacity.
It is also the largest subset of $\Omc$ with zero $\Csp$ capacity
such that $\Om \cup S$ is open.
Moreover,
\begin{align} \label{eq-S}
	S &= \{x\in\bdy \Om :
		\text{there is $r>0$ such that $\Csp(B(x,r)\setm \Om)=0$}\} \\
	 &= \{x\in\bdy \Om \setm \bdy \clOm:
		\text{there is $r>0$ such that $\Csp(B(x,r)\cap\bdy\Omega)=0$}\}.
\nonumber
\end{align}
\end{prop}

A direct consequence is that semiregularity is a local property of $\Om$,
and since regularity is a local property (by the Wiener criterion), so
is also strong irregularity.

It also follows that $S$ is removable in the sense
of Theorems~\ref{thm-remove} and~\ref{thm-remove-bdd},
with $E$ and $\Om \setm E$ therein replaced by  $S$ and $\Om$.
It follows from Remark~\ref{rmk-remove-sharp} that $S$ is the
largest such removable set, at least when the right-hand side $f$ is $0$
in Theorem~\ref{thm-remove}.

If $sp>n$, then all boundary points are regular (see Corollary~\ref{cor-sp>n=>reg}).
On the contrary, when $sp \le n$, the punctured ball (together with 
Proposition~\ref{prop-S-largest}) shows that there are semiregular boundary points.
Example~\ref{ex-strong-irr} 
exhibits
strongly irregular boundary points when $sp \le n$.

In the definition~\ref{ab} of strongly irregular points, 
the sequence $\{y_j\}_{j=1}^\infty$ is allowed to depend on $g$.
However, it is perhaps a bit surprising
 that there is a universal sequence $\{y_j\}_{j=1}^\infty$
such that \eqref{eq-strong-irr} holds 
for all \emph{bounded} $g \in \VspOm \cap C(\R^n)$,
see Theorem~\ref{thm-univ-seq}.
When $p=2$, this is true also for
\emph{unbounded} $g \in \VstOm \cap C(\R^n)$,
see Theorem~\ref{thm-str-irr-linear}.
(In the nonlinear case we do not know 
if this holds for unbounded $g$.)

Other characterizations
of strongly irregular  and semiregular boundary points
are proved in Sections~\ref{sect-trich} and~\ref{sect-strong-irr}.
In particular, we show that (semi)regularity 
and strong irregularity
can equivalently be defined using Perron solutions and 
can conveniently be tested using only
one function in the following way.
Theorem~\ref{thm-str-irr-intro} follows from
Theorems~\ref{thm-main-reg}, \ref{thm-semi-d} and~\ref{thm-str-irr}.
Part~\ref{dx0-reg} is from 
our earlier paper~\cite[Theorem~1.7]{BBK1}.

\begin{thm}\label{thm-str-irr-intro}
 Let $x_0 \in \bdy \Om$  and $d_{x_0}=\min\{1,|x-x_0|\}$.
 Then 
 \begin{enumerate}
\item \label{dx0-reg}
$x_0$ is regular if and only if 
\begin{equation*}
    \lim_{\Om \ni x \to x_0} H d_{x_0}(x) =0.
\end{equation*}
\item 
$x_0$ is strongly irregular if and only if 
the limit
\begin{equation*}
    \lim_{\Om \ni x \to x_0} H d_{x_0}(x) 
    \quad \text{does not exist}.
\end{equation*}
\item $x_0$ is semiregular if and only if 
\begin{equation*}
    \liminf_{\Om \ni x \to x_0} H d_{x_0}(x) >0.
\end{equation*}
\end{enumerate}
Moreover, the set of strongly irregular points is exactly the set $\itoverline{R} \setm R$,
where $R$ is the set of all regular boundary points.
 \end{thm}

There are two main ingredients needed to prove Theorem~\ref{thm-trich}:
The Kellogg property (Theorem~\ref{thm-kellogg})
from our earlier paper~\cite{BBK1}, 
which says that the set
of irregular boundary points has $\Csp$-capacity zero,
and the following new removability result.

\begin{thm} \label{thm-remove}
Assume that $\Om$ is a nonempty open set, which may be unbounded.
Let $E\subset\Om$ be a relatively closed set with $\Csp(E)=0$.
Assume that $u\in \Vsp(\Omega \setminus E)$ is a 
solution\/ \textup(resp.\ supersolution\/\textup)
of $\LL u=f$ in $\Om\setm E$ with $f\in \Vsp_0(\Om)^*$.
Then $u$ is a 
\textup(super\/\textup)solution
of $\LL u=f$ in $\Om$.
\end{thm}

In order to prove Theorem~\ref{thm-trich}
it is enough to have the removability result for solutions
of $\LL u=0$.
We deduce it in this more general form for supersolutions and with
right-hand side since this may be of independent interest
and the proof is essentially the same.
In Remark~\ref{rmk-remove-sharp} we show that this result is sharp
at least when~$\Om$ is bounded and the right-hand side is $f\equiv 0$.

For linear local equations
in an abstract setting, the trichotomy was developed
in detail by Luke\v{s}--Mal\'y~\cite{lukesmaly}.
For the local nonlinear \p-Laplace equation $\Delta_pu=0$,
the trichotomy as 
in Theorem~\ref{thm-trich} was proved by Bj\"orn~\cite{ABclass}.
Removability results as in Theorem~\ref{thm-remove}
for such equations are a classical topic and 
can be found already in 
Serrin's pioneering paper~\cite[Chapter~II]{Ser} from 1964.
See also Heinonen--Kilpel\"ainen--Martio~\cite[Theorems~7.35 and~7.36]{HeKiMa}
for removability results for weighted degenerate and singular nonlinear equations.

The following is the outline of the paper.
In Section~\ref{sec-sobolev} we recall some function spaces
and the notion of (weak) solutions of $\LL u=f$
with right-hand sides $f$. 
In Section~\ref{sect-cap}, 
we also recall the definition of the (Sobolev) capacity $\Csp$
(as well as the condenser capacity $\csp$ 
which we need in our analysis), 
prove its countable subadditivity and
review several known properties.

In Section~\ref{sect-rem} we turn to removable singularities
and prove Theorem~\ref{thm-remove} and discuss
its sharpness in Remark~\ref{rmk-remove-sharp}.

In Section~\ref{sect-DP}, we summarize the results
on the Dirichlet problem and regular boundary points,
which provide tools for the developments in the sequel.
In Section~\ref{sect-trich} we turn to the main topic of the paper,
the trichotomy classification mentioned above. 
Theorem~\ref{thm-trich} is proved here
and many characterizations of semiregular boundary points are given.
In Section~\ref{sect-strong-irr} we characterize
strongly irregular boundary
points in several ways.

We end the paper, in Section~\ref{sect-different-sp},
with deducing the following result
on   how semiregularity varies with $s$ and $p$.
Here $s=1$ is included, which relates the results  for nonlocal equations
to those for local equations.

\begin{prop} \label{prop-semireg-inclusion-intro}
Consider $0 < s_j\le1 < p_j$, $j=1,2$.
Then the
implication
\begin{equation*}
\text{$x_0$ is semiregular for $(s_1,p_1)$}
\imp
\text{$x_0$ is semiregular for $(s_2,p_2)$}
\end{equation*}
holds, for all bounded open sets $\Om$ with $x_0 \in \bdy \Om$,
if and only if
any of the following  mutually disjoint cases holds\/\textup{:}
\begin{enumerate}
\item
  $s_1 p_1 > n$,
\item
  $s_2 p_2 = s_1 p_1\le n$ and $p_1 \le p_2$,
\item
  $s_2 p_2 < s_1 p_1 \le n$.
\end{enumerate}  
\end{prop}

\begin{ack}
A.~B. and J.~B. were supported by the Swedish Research Council,
  grants 2020-04011 and 2022-04048, respectively.
MK was supported 
by the National Research Foundation of Korea (NRF) grant
funded by the Korean government (MSIT) (RS-2023-00252297).
\end{ack}

\section{Nonlocal equations and Sobolev spaces}\label{sec-sobolev}

\emph{Throughout the paper, we assume that  $1<p<\infty$, 
that $\Om \subset \Rn$
is a nonempty  open set, $n \ge 1$, and that $k$
 is a symmetric measurable kernel satisfying \eqref{eq-comp-(x,y)}.
In Sections~\ref{sect-cap}
and \ref{sect-DP}--\ref{sect-different-sp},
$\Om$ will also be assumed to be bounded.
Except for Section~\ref{sect-different-sp}, we 
assume that
$0<s<1$ throughout the paper.
}

\medskip

In order to consider weak solutions of the equation
\begin{equation}\label{eq-Lu=0}
\LL u=0 \quad\text{in}~\Om,
\end{equation}
we need 
suitable function spaces.

For a measurable function $u: \Om \to \eR:=[-\infty,\infty]$ (which is
finite a.e.) we consider   
the fractional seminorm  
\begin{equation*}
  [u]_{\Wsp(\Om)}=
\biggl(  \int_{\Omega}\int_{\Omega} \frac{|u(x)-u(y)|^p}{|x-y|^{n+s p}} 
      \, dy\, dx\biggr)^{1/p}. 
\end{equation*}
The fractional Sobolev space $\Wsp(\Omega)$ consists of the functions $u$ such that
the norm  
\begin{equation*}
\|u\|_{\Wsp(\Om)}^p:=\|u\|_{L^p(\Om)}^p+[u]_{\Wsp(\Om)}^p < \infty.
\end{equation*}
The above spaces and (semi)norms go under various names, 
such as fractional Sobolev, 
Gagliardo--Nirenberg, Sobolev--Slobodetski\u{\i} and Besov.

By $W^{s, p}_{\mathrm{loc}}(\Omega)$ we denote the space of functions that belong to $W^{s, p}(G)$ 
for every open $G \Subset \Omega$. 
As usual, by $E \Subset \Om$ we mean that $\itoverline{E}$
is a compact subset of $\Om$.
We refer the reader to Di Nezza--Palatucci--Valdinoci~\cite{DNPV12} for properties of these spaces.

As \eqref{eq-Lu=0} is a nonlocal equation, we 
need a function space that captures integrability of
functions in the whole of $\Rn$.
The \emph{tail space} $L^{p-1}_{sp}(\R^n)$ is given by
\begin{equation*}
L^{p-1}_{sp}(\R^n) = \biggl\lbrace u \text{ measurable}: 
  \int_{\R^n} \frac{|u(y)|^{p-1}}{(1+|y|)^{n+sp}} \,dy < \infty \biggr\rbrace,
\end{equation*}
see Kassmann~\cite{Kass11}  (for $p=2$) and Di Castro--Kuusi--Palatucci~\cite{DCKP16}.
For measurable functions $u, v: \R^n \to \eR$ 
we define the quantity
\begin{equation*}
\E(u,v)=\int_{\R^n} \int_{\R^n} |u(x)-u(y)|^{p-2} (u(x)-u(y))(v(x)-v(y)) k(x, y) \,dy\,dx,
\end{equation*}
provided that it is finite. Note that $\E(u, v)$ is well defined for 
$u \in \Wsp_{\mathrm{loc}}(\Om) \cap L^{p-1}_{sp}(\R^n)$ and $v \in C_c^\infty(\Om)$
and that it is linear in the second argument.
Here and later, $C_c^\infty(\Om)$ denotes
the space of $C^\infty$ functions with compact support in $\Om$.
Similarly, $\Lipc(\Om)$ denotes the space of Lipschitz 
functions with compact support in $\Om$.
The \emph{support} of a function $u$ is 
$\supp u =\overline{\{x:u(x)\ne0\}}$.
As usual, we let $u_\limplus=\max\{u,0\}$ and $u_\limminus=\max\{-u,0\}$.

We use the following definition of solutions and supersolutions of \eqref{eq-Lu=0}, 
which is standard by now.

\begin{deff} \label{def-supersol}
Let $u \in W^{s, p}_{\mathrm{loc}}(\Om) \cap L^{p-1}_{sp}(\R^n)$
and let $f$ be a measurable function in $\Omega$. 
Then $u$ is a (weak) \emph{solution} (resp.\ \emph{sub/supersolution}) 
of $\LL u=f$ in $\Om$ 
if $\langle f, \varphi \rangle:= \int_{\Omega}f(x) \varphi(x) \,dx$ is well defined and
\begin{equation} \label{eq-weaksoln}
\E(u, \varphi) \ge  \langle f, \varphi \rangle
\end{equation}
for all  $\varphi \in C_c^{\infty}(\Omega)$
(resp.\ for all nonpositive/nonnegative $\varphi \in C_c^{\infty}(\Omega)$).

If $u \in C(\Om)$ is a 
solution of $\LL u =0$ in $\Om$, 
then $u$ is \emph{$\LL$-harmonic} in $\Om$.
\end{deff}

Here we have included equations $\LL u =f$ with right-hand sides $f$.
In Section~\ref{sect-rem} we will deduce a removability result
for such equations, but in the rest of the paper we only consider equations
$\LL u =0$, with $0$ in the right-hand side.

If $u$ is a solution 
of $\LL u =0$, then there is an $\LL$-harmonic function
$v$ such that $v=u$ a.e.,
see e.g.\ Di Castro--Kuusi--Palatucci~\cite[Theorem~1.4]{DCKP16}.
While their definition of 
solutions to $\LL u =0$ differs slightly from ours, 
the same argument applies and shows that
every solution in our sense 
admits a representative that is
locally H\"older continuous in $\Omega$.
The notion of supersolutions to $\LL u =0$,
as defined by Korvenp\"a\"a--Kuusi--Palatucci~\cite{KKP17},
only requires that 
$u \in W^{s, p}_{\mathrm{loc}}(\Om) $ and 
$u_\limminus \in L^{p-1}_{sp}(\R^n)$.
However, \cite[Lemma~1]{KKP17} shows that their
definition is equivalent to the one adopted here.

It is convenient to deal with a larger class of test functions than $C_c^\infty(\Om)$.
For this purpose, we will use the spaces 
\begin{equation*}
\Vsp(\Omega) := \biggl\lbrace u: \R^n \to \eR : 
u|_{\Omega} \in L^p(\Omega) \text{ and }  
\frac{|u(x)-u(y)|}{|x-y|^{n/p+s}} \in L^p(\Omega \times \R^n) \biggr\rbrace,
\end{equation*}
equipped with the norm
\begin{align*}
\|u\|_{V^{s, p}(\Omega)} 
&:= \bigl( \|u\|_{L^p(\Omega)}^p + [u]_{V^{s, p}(\Omega)}^p \bigr)^{1/p} \\
&:= \biggl( \int_{\Omega} |u(x)|^p \,dx + \int_{\Omega} \int_{\R^n} \frac{|u(x)-u(y)|^p}{|x-y|^{n+sp}} \,dy \,dx \biggr)^{1/p},
\end{align*}
and
\begin{equation*}
V^{s, p}_0(\Omega) := \overline{C_c^{\infty}(\Omega)}^{V^{s, p}(\Omega)}.
\end{equation*}

With this definition, 
Proposition~2.2 in \cite{BBK1} shows that
a function $u\in \Wsploc(\Om) \cap L^{p-1}_{sp}(\R^n)$
is a (super)solution of $\LL u=f$ in $\Om$ if and only if \eqref{eq-weaksoln}
holds for all  (nonnegative) $\phi\in\VspOm$ with $\supp\phi\Subset\Om$.
By writing $\phi=\phi_\limplus - \phi_\limminus$ and using the
linearity of $\E(\cdot\,,\cdot)$ with respect to the second argument,
it is then easy to see that $u$ is a solution of $\LL u =f$
if and only if it is both a sub- and  supersolution of $\LL u =f$.

Note that $\Vsp(\R^n)=\Wsp(\R^n)$
and that
\[
\Lipc(\Om) \subset \Vspo(\Om) \subset
\Vsp(\Om) \subset
\Wsp(\Om) \cap L^{p-1}_{sp}(\R^n).
\]

See 
\cite[Remark~2.2]{KLL23}, \cite[Section~2]{BBK1}
and the references therein for further remarks 
on the spaces $\Vsp_0(\Om)$ and $\Vsp(\Om)$,
which were denoted by $\Wsp_0(\Om)$ and $\Vsp(\Om|\Rn)$
in~\cite{KLL23}.
 The space $V^{s,2}(\Om)$ was (with $p=2$) introduced by
 Servadei--Valdinoci~\cite{SV14} and independently by
 Felsinger--Kassmann--Voigt~\cite{FKV15}.

In this paper, there are many places where the functions need to be 
defined pointwise everywhere in a given set, and not just a.e. 
For convenience, we will therefore assume that all functions are defined
pointwise everywhere.
When saying that $u \in \VspoOm$ we will always assume that $u \equiv 0$ 
outside $\Om$. 
To handle complications with $\infty - \infty$ 
when letting $w=u-v$ 
we will say that $w(x)=0$ whenever 
$u(x)=v(x)$ even when
$u(x)= v(x)=\pm \infty$.

\begin{lem} \label{lem-Vsp-inclusion}
Let $G \subset \Omega$ be open sets. If $u \in V^{s, p}_0(G)$, 
then $u \in V^{s, p}_0(\Omega)$.
\end{lem}

\begin{proof}
Let $u_j \in C^\infty_c(G)$ be such that $u_j \to u$ in $V^{s, p}(G)$ as $j \to \infty$. Since $u_j \in C^\infty_c(\Omega)$ and
\begin{equation*}
\|u-u_j\|_{V^{s, p}(\Omega)} \leq \|u-u_j\|_{W^{s, p}(\Rn)} \leq 2\|u-u_j\|_{V^{s, p}(G)} \to 0
\quad \text{as }j \to \infty,
\end{equation*}
we conclude that $u \in V^{s, p}_0(\Omega)$.
\end{proof}

We shall also use the following well-known fractional Poincar\'e inequality
for functions $u\in \Lipc(\Om)$,
\begin{equation} \label{eq-frac-PI}
\|u\|_{W^{s,p}(\R^n)}^p \simle
(1+(\diam\Om)^{sp})[u]_{W^{s,p}(\R^n)}^p.
\end{equation}
It follows from the following simple argument 
for any $z\in\Rn$ with $\dist(z,\Om)=2\diam\Om$, 
\begin{equation*}
\int_{\Om}|u(x)|^p  \,  dx
\simle
(\diam\Om)^{sp} \int_{\Om}\int_{B(z,\diam\Om)} \frac{|u(x)-u(y)|^p}{|x-y|^{n+s p}}  \, dy\, dx.
\end{equation*}
See Maz{\cprime}ya--Shaposhnikova~\cite{MS} and Ponce~\cite{Ponce} 
for much more general  Poincar\'e inequalities.
Here and later, we write $a \simle b$ 
if there is an implicit comparison constant $C>0$ such that $a \le Cb$, 
and $a \simeq b$ if $a \simle b \simle a$.

\section{Capacities}\label{sect-cap}

\emph{In the rest of the paper, except for Section~\ref{sect-rem},
we assume that $\Om$ is a bounded nonempty open set.}

\medskip

There are two 
types of capacities used in this paper: 
the condenser capacity $\cpt$, which appears in the Wiener criterion,
and the Sobolev capacity $\Cpt$, which is convenient when dealing
with function spaces and in the Kellogg property, since it involves only
one set, unlike $\cpt$.
In contrast to Kim--Lee--Lee~\cite{KLL23}
we will need capacities also for noncompact sets.
We therefore make the following definitions.

\begin{deff} \label{deff-cpt}
The \emph{condenser capacity} of a compact set $K \Subset \Om$ is given by
\[
    \cpt(K,\Om) = \inf_u {[u]_{W^{s,p}(\R^n)}^p},
\]
where the infimum is taken over all  $u\in C_c^\infty(\Om)$
such that  $u \ge 1$  on $K$.
\end{deff}

\begin{deff}
The \emph{Sobolev capacity} of a compact set $K \Subset \Rn$ is defined by
\begin{equation}  \label{eq-def-norm-cap}
 \Cpt(K)   =\inf_u {\|u\|_{\Wsp(\Rn)}^p},
\end{equation}
where the infimum is taken over all $u\in C_c^\infty(\Rn)$ such that 
$u \ge 1$ on $K$.
\end{deff}

Both capacities are then extended first to open and then to arbitrary sets 
in the usual way as follows:
For open sets $G$,
\begin{equation}   \label{eq-cap-G}
\begin{aligned}
\cpt(G,\Om) & = \sup_{\substack{K \text{ compact}\\ K \Subset G}} \cpt(K,\Om), && \text{if $G \subset \Om$}, \\
\Cpt(G) &= \sup_{\substack{K \text{ compact}\\ K \Subset G}} \Cpt(K), && \text{if $G \subset \Rn$}.
\end{aligned}
\end{equation}
Similarly, for arbitrary  sets~$E$,
\begin{equation}  \label{eq-cap-E}
\begin{aligned}
\cpt(E,\Om) &= \inf_{\substack{G \text{ open}\\ E \subset G \subset \Om}} \cpt(G,\Om), && \text{if $E \subset \Om$}, \\
\Cpt(E) &= \inf_{\substack{G \text{ open}\\ E \subset G}} \Cpt(G), && \text{if $E \subset \Rn$}.
\end{aligned}
\end{equation}

It is worth noticing that~\eqref{eq-cap-G} and~\eqref{eq-cap-E} do not change the 
definition of capacity for compact sets, 
see \cite[Section~5]{BBK1}.

We will need some results
from~\cite[Section~5]{BBK1}, see therein for a more
detailed discussion.
The following result
shows that 
the two capacities $\cpt$ and $\Cpt$
have the same zero sets.

\begin{lem} \label{lem-cp-Cp}
\textup{(\cite[Proposition~5.4]{BBK1})}
Let $E  \Subset \Om$.
Then  
$\Cpt(E)=0$ if and only if  $\cpt(E,\Om)=0$.
\end{lem}

\begin{lem} \label{lem-sp<=n}
\textup{(\cite[Lemma~5.5]{BBK1})}
Let $x_0 \in \Om$.
Then the following are equivalent\/\textup:
\begin{enumerate}
\item
$sp \le n$,
\item
$\Csp(\{x_0\})=0$, 
\item
$\csp(\{x_0\},\Om)=0$.
\end{enumerate}
\end{lem}

\begin{lem}\label{lem-zero-cap}
\textup{(\cite[Lemma~5.6]{BBK1})}
If $E\subset \R^n$ is measurable, then
$\Cpt(E) \ge |E|$, where $|\cdot|$ denotes Lebesgue measure.
\end{lem}

The following result is  
convenient when calculating capacities.

\begin{prop} \label{prop-Wsp-Csp}
\textup{(\cite[Lemma~5.7]{BBK1})}
Let $E \subset\Rn$.
Then
\begin{equation*} 
   \Cpt(E) = \inf_{u} {\|u\|^p_{\Wsp(\Rn)}},
\end{equation*}
where the infimum is taken over all $u \in \Wsp(\Rn)$ such that
$0 \le u \le 1$ everywhere and $u=1$ 
in  an open set containing $E$.
\end{prop}

A simple consequence of Proposition~\ref{prop-Wsp-Csp} is that
$\Cpt$ is countably subadditive. 
This does not seem to be so easy to prove
directly from the three-step definition of the Sobolev capacity.
We  include a  short proof for the reader's convenience.

\begin{prop} \label{prop-Cpt-subadd}
If $E=\bigcup_{i=1}^\infty E_i \subset \R^n$, then 
\begin{equation} \label{eq-Cpt-subadd}
      \Cpt(E)
          \le \sum_{i=1}^\infty \Cpt(E_i).
\end{equation}

If also $E \subset \Om$, then 
\begin{equation} \label{eq-csp-subadd}
      \csp(E,\Om)
          \le \sum_{i=1}^\infty \csp(E_i,\Om).
\end{equation}
\end{prop}

\begin{proof}
We first consider \eqref{eq-Cpt-subadd}.
We may assume that the right-hand side is finite.
Let $\eps>0$.
By 
Proposition~\ref{prop-Wsp-Csp}, there is 
for each $i=1,2,\dots$\,, 
a function $u_i \in \Wsp(\Rn)$ such that 
$u_i = 1$  in a neighbourhood of $E_i$, $0 \le u \le 1$ everywhere
and 
\[
\|u_i\|_{W^{s, p}(\R^n)}^p < \Cpt(E_i)+ \frac{\eps}{2^i}.
\]
Let $u=\sup_i u_i$.
Then $u = 1$ in a neighbourhood of $E$.
Moreover, for $x,y \in \Rn$, 
\[
  |u(x)-u(y)|^p
  \le \sup_i |u_i(x)-u_i(y)|^p
  \le \sum_{i=1}^\infty |u_i(x)-u_i(y)|^p
\]
and similarly,
$|u(x)|^p =   \sup_i |u_i(x)|^p 
  \le \sum_{i=1}^\infty |u_i(x)|^p$.
Hence, by Proposition~\ref{prop-Wsp-Csp} again, 
\begin{align*}
      \Cpt(E) 
      & \le \|u\|_{W^{s,p}(\R^n)}^p
       \le \sum_{i=1}^\infty \|u_i\|_{W^{s,p}(\R^n)}^p \\
      & 
           < \sum_{i=1}^\infty \Bigl( \Cpt(E_i)+ \frac{\eps}{2^i}\Bigr) 
           =   \sum_{i=1}^\infty \Cpt(E_i) + \eps.
\end{align*}
Letting $\eps \to 0$ completes the proof of \eqref{eq-Cpt-subadd}.

The proof of \eqref{eq-csp-subadd}
is similar, but
using  \cite[Proposition~6.2]{BBK1} instead
of Proposition~\ref{prop-Wsp-Csp}.
\end{proof}

\section{Removable singularities}
\label{sect-rem}

\emph{In contrast to most of the paper, in this section
we assume that $\Om$ is a nonempty open set that may be unbounded.}

\medskip

\begin{proof}[Proof of Theorem~\ref{thm-remove}]
We first consider the case for supersolutions.
Let $0 \le \phi\in C^\infty_c(\Om)$ be a test function.
Since $\Csp(E)=0$ and $E\cap \supp\phi$ is compact, 
for each $j=1,2,\dots$\,, there are open sets $G_j$ and $G_j'$ such that
\[
E\cap \supp\phi \Subset G_j \Subset G_j' \quad \text{and} \quad  \Csp(G_j')<1/j.
\]
Hence $\Csp(\clG_j) <1/j$, and thus by \eqref{eq-def-norm-cap},
there  are
$\eta_j'\in C_c^\infty(\R^n)$ such that $\eta_j'\ge1$ on $\clG_j$
and $\|\eta_j'\|_{W^{s,p}(\R^n)}< 1/j$.
Thus there is a mollification $\eta_j \in C_c^\infty(\R^n)$
of $\min\{\eta_j',1\}_\limplus$ such that 
$\eta_j=1$ on
$E\cap \supp\phi$, $0\le \eta_j \le 1$ in $\Rn$
and 
$\|\eta_j\|_{W^{s,p}(\R^n)}< 1/j$.
After taking a subsequence if necessary,
we can assume that also $\eta_j\to0$ a.e.\ as $j\to\infty$.
Then $\phi_j:=(1-\eta_j)\phi\in C^\infty_c(\Om\setm E)$ and hence
\begin{equation}   \label{eq-weak-sol-phij}
\int_{\Rn}\int_{\Rn} |u(x)-u(y)|^{p-2} (u(x)-u(y))(\phi_j(x)-\phi_j(y)) k(x, y)
      \, dy\, dx \ge
\langle f, \phi_j \rangle.
\end{equation}
Note that $\phi-\phi_j=\phi\eta_j$ and
\[
  |(\phi\eta_j)(x)-(\phi\eta_j)(y)| \le |(\phi(x)-\phi(y))\eta_j(x)|
       + |\phi(y)(\eta_j(x)-\eta_j(y))|.
\]
Since $\phi$ is bounded, we therefore get that $\phi_j\to\phi$ in $L^p(\R^n)$ and
\begin{equation}   \label{eq-[]-to0}
  [\phi\eta_j]_{W^{s,p}(\R^n)}^p
  \le 
  2^p\biggl( \int_{\Rn}\eta_j(x)^p \int_{\Rn} \frac{|\phi(x)-\phi(y)|^p}{|x-y|^{n+sp}} 
      \, dy\, dx   
    +
 [\eta_j]_{W^{s,p}(\R^n)}^p  \max_{\R^n} |\phi|^p 
\biggr).
\end{equation}
The last term clearly tends to 0 as $j\to\infty$ by the choice of $\eta_j$.
Since 
$\eta_j\to0$ a.e.\ and
\[
\int_{\Rn} \frac{|\phi(x)-\phi(y)|^p}{|x-y|^{n+sp}}  \, dy \in L^1(\R^n),
\]
dominated convergence implies that also the integral term in~\eqref{eq-[]-to0}
tends to 0 as $j\to\infty$.
Hence
\begin{equation}   \label{eq-semi-to0}
[\phi-\phi_j]_{W^{s,p}(\R^n)}^p = [\phi\eta_j]_{W^{s,p}(\R^n)}^p \to 0,
\quad \text{as } j\to\infty.
\end{equation}
Since $f\in \Vsp_0(\Om)^*$ and $\phi-\phi_j\in \Vsp_0(\Om)$,
it follows that the right-hand side
in~\eqref{eq-weak-sol-phij} converges to
\(
\langle f, \varphi \rangle.
\)
As for the left-hand side in~\eqref{eq-weak-sol-phij}, an application of
H\"older's inequality with respect to $dy\,dx$, together with the fact that $\phi=0$  outside
$\Om$, shows that
\begin{align*}
&\biggl| \int_{\Rn}\int_{\Rn} |u(x)-u(y)|^{p-2} (u(x)-u(y))((\phi\eta_j)(x)-(\phi\eta_j)(y)) k(x, y)
                 \, dy\, dx \biggr| \\
&\quad \le 2\int_{\Om}\int_{\Rn} |u(x)-u(y)|^{p-1} |(\phi\eta_j)(x)-(\phi\eta_j)(y)| k(x, y) 
                 \, dy\, dx \\
  &\quad \le 2\La [u]_{\Vsp(\Om)}^{p-1} [\phi\eta_j]_{W^{s,p}(\R^n)} \to 0,
\quad \text{as } j\to\infty,
\end{align*}
by~\eqref{eq-semi-to0} and the fact that $u\in \Vsp(\Omega \setminus E) = \Vsp(\Omega)$ (since $|E|=0$, by \cite[Lemma~5.6]{BBK1}).
Hence
\[
  \E(u,\phi) = \lim_{j \to \infty} \E(u,\phi_j) 
\ge  \lim_{j \to \infty} \langle f,\phi_j\rangle
=    \langle f,\phi\rangle.
\]

Finally, assume that $u$ is a solution of $\LL u =f$ in $\Om \setm E$.
Then it is both a supersolution and a subsolution in $\Om \setm E$, and thus 
in $\Om$ by the above.
Hence, it is a solution of $\LL u =f$ in $\Om$.
\end{proof}

\begin{remark} \label{rmk-remove-sharp}
Theorem~\ref{thm-remove} is sharp, at least if $\Om$ is bounded and $f \equiv 0$. 
Indeed, assume that $\Om$ is bounded and let 
$E \subset \Om$ be a relatively closed set with $\Csp(E)>0$.
Then $E=\bigcup_{j=1}^\infty K_j$, where $K_j=\{x \in E : \dist(x,\Omc) \ge  1/j\}$.
It thus follows from 
the countable subadditivity of $\Csp$
(see Proposition~\ref{prop-Cpt-subadd})  
that 
there is $j$ such that the compact subset $K=K_j \Subset E$ has $\Csp(K)>0$.
We consider the capacitary $\LL$-potential 
$u$ of $K$ in $\Omega$, 
i.e.\ $u=H_{\Omega \setminus K}\psi$ 
(in the notation of Theorem~\ref{thm-DP} below), 
where $\psi \in C^\infty_c(\Omega)$ is such that $\psi=1$ on $K$. 
In particular, $u$ is a bounded solution of $\LL u =0$ in $\Om \setm E$
and a supersolution of $\LL u =0$ in $\Om$.

Suppose that $u$ is a solution of $\LL u=0$ in $\Omega$. 
We may assume that $u$ is  $\LL$-harmonic in $\Omega$. 
Lemma~\ref{lem-Vsp-inclusion} shows that $u-\psi \in V^{s,p}_0(\Omega)$. 
Then $u=H_\Omega\psi$. 
It now follows from the 
comparison principle (Theorem~\ref{thm-comparison}) 
that $u \equiv 0$ in $\Rn$. However, \cite[Lemma~2.16(iii)]{KLL23} shows that 
$\csp(K, \Omega) \simeq
\E(u, u) = 0$, which contradicts  the assumption 
$C_{s,p}(K)>0$ by Lemma~\ref{lem-cp-Cp}.
Hence $u$ cannot be
a solution of $\LL u=0$ in $\Om$,
which shows the sharpness for solutions.

Let now $v=-u$ which is also a solution of 
$\LL v=0$ in $\Om \setminus E$.
Since $u$ is a supersolution, but not a solution, it cannot
be a subsolution of $\LL u=0$ in $\Om$.
Hence $v=-u$ is a solution of $\LL v=0$ in $\Om \setm E$ 
but not a supersolution of $\LL v=0$ in $\Om$.

This also shows that Theorem~\ref{thm-remove-bdd} below
is sharp, at least when $\Om$ is bounded.
\end{remark}

To prove the trichotomy for Perron solutions we need
the following removability result for bounded $\LL$-harmonic functions.
It is a special case of the main result in 
Kim--Lee~\cite{KL-rem}.
Unlike in Theorem~\ref{thm-remove}, here $u$ is not assumed to
belong to $\Vsp(\Om\setm E)$, it only needs to be bounded.

\begin{thm} \label{thm-remove-bdd}
\textup{(\cite[Theorem~1.1]{KL-rem})}
Assume that $\Om$ is a nonempty open set, which may be unbounded.
Let $E\subset\Om$ be a relatively closed set with $\Csp(E)=0$.
Assume that $u$ is a bounded 
$\LL$-harmonic function 
in $\Om\setm E$.
Then $u$ is a solution
of $\LL u=0$ in $\Om$,
and there is an $\LL$-harmonic function $v$ in $\Om$
such that $v=u$ in $\Rn \setm E$ and
$v=u$ a.e.\ in $\Rn$.
\end{thm}

\section{The Dirichlet problem and regular boundary points}
\label{sect-DP}

\emph{In the rest of the paper, 
we assume that $\Om$ is a bounded nonempty open set.}

\medskip

For functions $g \in \VspOm$ one can solve
the Dirichlet problem in the Sobolev sense. 
To make it precise we define the Sobolev solution $Hg$ as follows.
Its existence and uniqueness follow from 
Korvenp\"a\"a--Kuusi--Palatucci~\cite[Theorem~9]{KKP17} and
Kim--Lee~\cite[Theorem~4.9]{KL-orlicz}, cf.\ \cite[Theorem~3.3]{BBK1}.
Here and in the rest of the paper, 
we consider 
the equation $\LL u=0$ with zero right-hand side.

\begin{thm} \label{thm-DP}
\textup{(\cite[Theorem~4.9]{KL-orlicz} and \cite[Theorem~9]{KKP17})}
Let $g \in \VspOm$.
Then there is a unique function $Hg := H_\Om g:\R^n \to \eR$ with the following properties\/\textup{:}
\begin{enumerate}
\item
  $Hg$ is $\LL$-harmonic in $\Om$, i.e.\ a continuous solution of $\LL u=0$,
\item
  $Hg-g \in \VspoOm$, and in particular
  $Hg \equiv g$ outside $\Om$.
\end{enumerate}
\end{thm}

\begin{thm}  \label{thm-comparison}
\textup{(Comparison principle for $H$-solutions, \cite[Corollary~3.6]{BBK1})}
Let $g_1, g_2 \in \Vsp(\Om)$ be such that 
$(g_1-g_2)_\limplus \in \Vsp_0(\Om)$.
Then $Hg_1 \le Hg_2$ in $\R^n$.
\end{thm}

\begin{deff}
  A boundary point $x_0 \in \bdy \Om$ is \emph{regular} (with respect to $\LL$
and $\Om$)
  if 
  \[
    \lim_{\Om \ni x \to x_0} H g(x)=g(x_0)
    \quad \text{for every } g \in \VspOm \cap C(\R^n).
  \]  
  Otherwise we say that $x_0$ is \emph{irregular}.
\end{deff}  

This definition was used by Kim--Lee--Lee~\cite{KLL23}.
It
was called Sobolev regularity in
Bj\"orn--Bj\"orn--Kim~\cite{BBK1},
but one of the main results in~\cite[Theorem~1.7]{BBK1}
showed that it is equivalent to regularity for 
Perron solutions. 
We therefore call it just regular here.

In order to define Perron solutions we need $\LL$-superharmonic functions,
which we next define.
We follow Korvenp\"a\"a--Kuusi--Palatucci~\cite[Definition~1]{KKP17},
see also~\cite[Section~8]{BBK1}.

\begin{deff} \label{def:superharmonic}
A measurable function $u: \R^n \to \eR$ is \emph{$\LL$-superharmonic}  
in $\Omega$ if it satisfies the following properties:
\begin{enumerate}
\item
$u < \infty$ almost everywhere in $\R^{n}$ and $u>-\infty$ everywhere in $\Omega$,
\item
$u$ is lower semicontinuous  in $\Omega$,
\item
for each open set $G \Subset \Omega$ and each solution $v \in C(\clG)$ of $\mathcal{L}v=0$ 
in $G$ satisfying $v_\limplus \in L^{\infty}(\R^{n})$ and $v \le u$ on  $\Gc$,
it holds that $v \le u$ in $G$,
\item
$u_\limminus \in L^{p-1}_{sp}(\R^{n})$.
\end{enumerate}
A function $u$ is  \emph{$\LL$-subharmonic} in $\Omega$ if $-u$ is $\LL$-superharmonic in $\Omega$.
\end{deff}

We are now  ready to define the Perron solutions to the Dirichlet problem. 
We follow the definition in our paper~\cite[Definition~1.1]{BBK1}.
For bounded continuous $g$ it coincides with the definition in the earlier
paper by Lindgren--Lindqvist~\cite[Section~4]{LL17}
(who however only considered the fractional \p-Laplacian with kernel
$k(x,y)=|x-y|^{-n-sp}$).
For $g \in L^{p-1}_{sp}(\R^n)$,
it also (essentially) coincides
with the definition in 
Korvenp\"a\"a--Kuusi--Palatucci~\cite[Definition~2]{KKP17},
by Theorem~9.2 in~\cite{BBK1}.

\begin{deff} \label{def-Perron}
Let $g:\Omc \to \eR$. 
The \emph{upper class}  $\mathcal{U}_{g}=\mathcal{U}_{g}(\Om)$ 
of $g$ consists of all functions 
$u: \R^{n} \to (-\infty,\infty]$ such that
\begin{enumerate}
\renewcommand{\theenumi}{\textup{(\roman{enumi})}}%
\renewcommand{\labelenumi}{\theenumi}%
\item
$u$ is $\LL$-superharmonic in $\Omega$,
\item \label{P-b}
$u$ is bounded from below in $\Om$,
\item \label{P-c}
$\liminf_{\Om \ni y \to x} u(y) \geq g(x)$ for all $x \in \partial \Omega$,
\item \label{P-d}
$u \ge g$ in $\Omc$.
\end{enumerate}

The \emph{upper Perron solution} of $g$ is defined by
\[ 
    \uP g (x)= \uP_\Om g (x) = \inf_{u \in \UU_g}  u(x), \quad x \in \R^n,
\]
and the \emph{lower Perron solution} of $g$ by 
$\lP g = - \uP (-g)$.
(As usual, $\inf \emptyset = \infty$.)

We usually omit $\Om$ from the notation.
When $\lP g = \uP g$  we denote the common solution by $Pg$ and say that
$g$ is \emph{resolutive}.
\end{deff}

In this paper, we are primarily interested in Perron solutions
for bounded functions.
The following result summarizes the results on Perron solutions 
needed in this paper.
(Note that bounded measurable functions belong to the tail space
$L^{p-1}_{sp}(\Rn)$.)

\begin{thm} \label{thm-Perron}
\textup{(\cite[Theorems~1.2, 1.4 and~9.4]{BBK1})}
Let $g:\Rn \to \R$ 
be a bounded measurable function
which is continuous at every $x\in\bdy\Om$.
Then $g$ is resolutive.

If also $g \in \Vsp(\Om)$, then $Pg=Hg$.
\end{thm}

The main result in Kim--Lee--Lee~\cite{KLL23} is
the Wiener criterion for boundary regularity;
see also the correction for $sp>n$ in 
\cite[Remark~1.5]{KLL25}.
Several other characterizations 
of regular boundary points
were given in \cite{BBK1}.
The following result summarizes the characterizations needed in this paper.
In \ref{a-contx0} we slightly improve upon the corresponding
characterization 
from Theorem~4.5 in~\cite{BBK1}.
This will be important later on.

\begin{thm}\label{thm-main-reg}
\textup{(\cite[Theorem~1.1 and its proof]{KLL23} and~\cite[Theorems~1.7 and~4.5]{BBK1})}
  Let $x_0 \in \bdy \Om$.
Then the following are equivalent\/\textup{:}
\begin{enumerate}
\item \label{a-reg}
  $x_0$ is regular for $\Om$,
\item \label{a-Wiener}
\begin{equation} \label{eq-Wiener-w-csp}
  \int_0^1 \biggl(
  \frac{\csp(\itoverline{B(x_0,\rho)}\setm \Om,B(x_0,2\rho))}{\rho^{n-sp}}
          \biggr)^{1/(p-1)} \, \frac{d\rho}{\rho} = \infty,
\end{equation}
\item \label{a-d}
\[
    \lim_{\Om \ni x \to x_0} Hd_{x_0}(x)=0,
\quad \text{where }   d_{x_0}(x):=\min\{1,|x-x_0|\},
\]  
\item \label{a-contx0} 
  \[
    \lim_{\Om \ni x \to x_0} Hg(x)=g(x_0)
  \]  
for every 
$g \in \VspOm$ that is 
continuous at $x_0$,
\item \label{a-contx0-P}
  \begin{equation} \label{eq-a-contx0-P}
    \lim_{\Om \ni x \to x_0} \uP g(x)=g(x_0)
  \end{equation}
for every bounded $ g:\Omc\to \R$ that 
is continuous at $x_0$.
\end{enumerate}  
\end{thm}

\begin{proof}
\ref{a-reg}\eqv\ref{a-Wiener}\eqv\ref{a-d}\eqv\ref{a-contx0-P}.
This follows from Theorem~1.1 in~\cite{KLL23} and 
Theorem~1.7 in \cite{BBK1}.

\ref{a-contx0}\imp\ref{a-reg} 
This is trivial.

\ref{a-Wiener}\imp\ref{a-contx0}
In \cite{KLL23}, it is assumed
that $g \in C(\Rn) \cap \Vsp(\Om)$,
but the proof of the sufficiency of the Wiener criterion
(i.e.\ \ref{a-Wiener}\imp\ref{a-reg})
only uses
continuity in the second sentence on p.~1979,
and there it is enough to know that $g$ is continuous at $x_0$.
Hence \ref{a-Wiener}\imp\ref{a-contx0}.
\end{proof}

\begin{remark} \label{rmk-bdd-P-cannot-be-dropped}
The boundedness assumption in \ref{a-contx0-P} cannot be dropped.
To see this consider, e.g., $g(x)=e^{|x|} \in C(\Rn)$.
It follows from Theorem~14 in 
Korvenp\"a\"a--Kuusi--Palatucci~\cite{KKP17}
that every $\LL$-superharmonic belongs to $L^{p-1}_{sp}(\Rn)$,
and since $0 \le g \notin L^{p-1}_{sp}(\Rn)$, the upper class
$\UU_g(\Om) = \emptyset$ (for every $\Om$) and thus
$\uP g  \equiv \infty$
and \eqref{eq-a-contx0-P} fails.
This is regardless of if $x_0$ is regular or not.
Hence the boundedness assumption in \ref{a-contx0-P} cannot be dropped.

Similarly, 
the boundedness assumptions in 
Theorem~\ref{thm-semi-d}\ref{d-semi-P}, \ref{d-semi-P-gen},
Theorem~\ref{thm-str-irr}\ref{st-bdd-P}
and in \eqref{eq-univ-P}   
of Theorem~\ref{thm-univ-seq}
cannot be dropped.
\end{remark}

The following Kellogg property
will be crucial when proving Theorem~\ref{thm-trich}.

\begin{thm}\label{thm-kellogg}
\textup{(Kellogg property, \cite[Theorem~1.5]{BBK1})}
Let $I_\Om\subset\bdy\Om$ be the set of irregular boundary points for 
$\LL u=0$  in $\Om$.
  Then  $\Csp(I_\Om)=0$.
\end{thm}

The following result is a direct
consequence of 
the Kellogg property (Theorem~\ref{thm-kellogg}) and
Lemma~\ref{lem-sp<=n}.
It also follows directly from the Wiener
criterion 
in Kim--Lee--Lee~\cite{KLL23}, \cite{KLL25}.

\begin{cor} \label{cor-sp>n=>reg}
If $sp>n$, then every boundary point $x_0 \in \bdy \Om$ is
regular.
\end{cor}

When $sp \le n$, the centre of a punctured ball is a semiregular boundary point
by Proposition~\ref{prop-S-largest}, while the following example 
provides a strongly irregular boundary point.

\begin{example} \label{ex-strong-irr}
Assume that $sp  \le n$.
Let
$x_j=(\tfrac34 \cdot 2^{-j},0,\dots,0)$, $j=1,2\dots$\,.
By Lemma~\ref{lem-sp<=n}, points have zero capacity and thus
by \eqref{eq-cap-E} there are $0<r_j< \tfrac18 \cdot 2^{-j}$ such that
$\csp(B(x_j,2r_j),B(0,2^{-j}))<2^{-(j+1)^2}$, $j=1,2\dots$\,.
Let 
\[
    \Om =B(0,1) \setm \biggl( \{0\} \cup \bigcup_{j=1}^\infty K_j   \biggr),
\quad \text{where }
K_j=\itoverline{B(x_j,r_j)},\ j=1,2,\dots.
\]
Then boundary points $x \in \bdy \Om \setm \{0\}$
are regular by the Wiener criterion (Theorem~\ref{thm-main-reg}\ref{a-Wiener}).

Next, we will use the Wiener criterion to show that 
$0 \in \bdy\Omega$ is irregular.
Let $B_\rho=B(0,\rho)$.
If $2^{-k-1} \le \rho<2^{-k}$, then
by Proposition~\ref{prop-Cpt-subadd},
\begin{align*}
   \csp(\clB_\rho\setm \Om, B_{2\rho})
  &   \le \sum_{j=k}^\infty \csp(K_j, B_{2\rho})
   < \sum_{j=k}^\infty \csp(K_j, B_{2^{-k}}) \\
  & < \sum_{j=k}^\infty 2^{-(j+1)^2}
   < 2^{-k^2}.
\end{align*}
Hence
\begin{align*}
 \sum_{k=0}^\infty 
  \int_{2^{-k-1}}^{2^{-k}} \biggl(
  \frac{\csp(\clB_\rho\setm \Om,B_{2\rho})}{\rho^{n-sp}}
          \biggr)^{1/(p-1)} \, \frac{d\rho}{\rho} 
\le 
 \sum_{k=0}^\infty \biggl(  \frac{2^{-k^2}}{2^{-k(n-sp)}}          \biggr)^{1/(p-1)} 
< \infty,
\end{align*}
and thus $0$ is irregular, 
by the 
Wiener criterion (Theorem~\ref{thm-main-reg}\ref{a-Wiener}).
On the other hand, $0$ is not semiregular, by Proposition~\ref{prop-S-largest},
and it thus must be strongly irregular by the trichotomy (Theorem~\ref{thm-trich}).
This also follows from Theorem~\ref{thm-str-irr}\ref{st-R}.
\end{example}

\section{Semiregular  points}
\label{sect-trich}

Recall from the introduction that an irregular boundary point 
$x_0 \in \bdy \Om$ is semiregular if 
\begin{enumerate}
\renewcommand{\theenumi}{\textup{(\Roman{enumi})}}%
\renewcommand{\labelenumi}{\theenumi}%
\item \label{aa-sec}
For every $g \in \VspOm \cap C(\R^n)$ the limit 
$    \lim_{\Om \ni x \to x_0} H g(x)$ exists.
\end{enumerate}

\begin{proof}[Proof of Theorem~\ref{thm-trich}]
The three classes of boundary points are clearly mutually disjoint.
We consider two complementary cases.

\medskip

\emph{Case} 1.
\emph{There is $r>0$ such that $\Csp(B \setm \Om )=0$,
where $B=B(x_0,r)$.}
Let $g \in \Vsp(\Om) \cap C(\R^n)$ and $E = B \setm \Om$.
Then $G:=\Om \cup E=\Om \cup B$ is open and $E$ is relatively closed in $G$. 
By the removability Theorem~\ref{thm-remove}, 
$Hg:=H_{\Om}g$
is a solution of $\LL u=0$ in $G$.
Hence, there is an $\LL$-harmonic (and thus continuous)
function $u$ in $G$ such that $u=Hg$ a.e.
Since $Hg$ is continuous in $\Om$, we 
have $u \equiv Hg$ in $\Om$ and thus
\[
    \lim_{\Om \ni x \to x_0} Hg(x)
  =   \lim_{\Om \ni x \to x_0} u(x)
  = u(x_0),
\]
i.e.\ 
\ref{aa-sec}
holds and 
$x_0$ is either regular or semiregular.
However, it follows from Lemma~\ref{lem-cp-Cp} that
$\csp(\itoverline{B(x_0,\rho)} \setm \Om,B(x_0,2\rho))=0$ whenever $\rho<r$.
Thus the Wiener integral~\eqref{eq-Wiener-w-csp} is finite and 
by the Wiener criterion (see Theorem~\ref{thm-main-reg}), $x_0$ is irregular,
so it must be semiregular.

\medskip

\emph{Case} 2.
\emph{The capacity $\Cpt(B(x_0,r) \setm  \Om)>0$ for all $r>0$.}
For every $j=1,2,\dots$\,, we thus have
$\Cpt(B_j \setm \Om)>0$, where $B_j=B(x_0,1/j)$.
We shall find  regular (not necessarily distinct) 
boundary points $x_j \in  B_j \cap \bdy \Om$, $j=1,2,\dots$\,.
If $\Cpt(B_j \cap \bdy \Om)>0$, then it follows from
the Kellogg property (Theorem~\ref{thm-kellogg})
that there is a regular boundary point $x_j \in B_j \cap \bdy \Om$.

On the other hand,
if $\Cpt(B_j \cap \bdy \Om)=0$, then $G_j:=B_j \setm \clOm \ne \emptyset$.
Let $z \in \bdy G_j \cap B_j \subset \bdy \Om$ and 
$0 < t < \tfrac12  (1/j - |z-x_0|)$.
Then there is $z' \in G_j \cap B(z,t)$.
Let $x_j$ be a closest point in $\bdy \Om$ to $z'$.
Since $|x_j -z'| \le |z - z'| < t$, 
we see that  $x_j \in B_j$.
Moreover, $B(z',|x_j -z'|) \cap\Om=\emptyset$ and hence
there is   
an exterior ball 
at $x_j$ (with respect to $\Om$).  Thus,
by the Wiener criterion, $x_j$ is regular.

Let $g \in \VspOm \cap C(\R^n)$.
Since $x_j$ is regular, we can find $y_j \in B(x_j,1/j) \cap \Om$
so that $|H g(y_j)-g(x_j)|<1/j$.
It follows directly that $y_j \to x_0$ and $H g(y_j) \to g(x_0)$, 
as $j \to \infty$,
i.e.\ \eqref{eq-strong-irr} holds, and thus $x_0$ is either
regular or strongly irregular.
\end{proof}  

\begin{proof}[Proof of Proposition~\ref{prop-S-largest}]
The first identity in~\eqref{eq-S} follows directly
from the proof of Theorem~\ref{thm-trich},
because the semiregular points are exactly those appearing in 
Case~1.
The second identity in \eqref{eq-S} now follows 
since nonempty open sets have positive capacity, by Lemma~\ref{lem-zero-cap}.
This also shows that every strictly larger relatively open subset of 
$\bdy \Om \setm \bdy \clOm$ has positive capacity. 
Since also 
\[
    S=\bigcup  \{ B(y,r) \setm \Om :
y\in \Q^{n}, 0<r\in\Q \text{ and } \Csp(B(y,r) \setm \Om)=0 \},
\]
we see that $S$ is a countable union of sets of capacity zero,
and thus itself of capacity zero, by Proposition~\ref{prop-Cpt-subadd}.
It also follows from~\eqref{eq-S} that $\Om \cup S$ is open.

Let now $E \subset \Omc$ be a set with $\Csp(E)=0$ such that $\Om \cup E$
is open. 
Since nonempty open sets have positive capacity, by Lemma~\ref{lem-zero-cap},
we directly see that $E \subset \bdy \Om$.
As $\Om \cup E$ is open, \eqref{eq-S} implies that $E \subset S$   
and so $S$ is the largest such set.
\end{proof}

The following result characterizes semiregularity
in various ways.
In particular it shows that 
semiregularity can be tested using only the function $d_{x_0}$,
cf.\ Theorem~\ref{thm-main-reg}\ref{a-d}.
Note that $Hd_{x_0}=Pd_{x_0}$, by Theorem~\ref{thm-Perron},
and thus $Hd_{x_0}$ can be replaced by $Pd_{x_0}$ in condition~\ref{d-liminf} below.

\begin{thm}\label{thm-semi-d}
  Let $x_0 \in \bdy \Om$ and $d_{x_0}=\min\{1,|x-x_0|\}$.
Then the following are equivalent\/\textup{:}
\begin{enumerate}
\item \label{d-semi}
  $x_0$ is semiregular,
\item \label{d-notreg}
  $sp \le n$  and moreover
\begin{equation} \label{eq-semireg-bdd}
    \lim_{\Om \ni x \to x_0} H g(x) 
    \quad \text{exists and is finite for every  } g \in \VspOm,
\end{equation}
\item \label{d-liminf}
   \[
    \liminf_{\Om \ni x \to x_0} H d_{x_0}(x) >0,
  \]
\item \label{d-R}
$x_0 \notin \overline{\{x \in \bdy \Om : x \text{ is regular}\}}$,
\item \label{d-semi-P}
  $x_0$ is semiregular for Perron solutions, i.e.\
$x_0$ is irregular and 
\begin{equation} \label{eq-d-semi-P}
\lim_{\Om \ni x \to x_0} Pg(x) 
\quad \text{exists for every bounded $g \in C(\Omc)$,} 
\end{equation}
\item \label{d-semi-P-gen}
  $sp \le n$  and moreover
\begin{equation}    \label{lim-Pg-ex}
\lim_{\Om \ni x \to x_0} \uP g(x) 
\quad \text{exists for every bounded measurable $g:\Omc \to \R$.}
\end{equation}
\end{enumerate}
\end{thm}

Note that, by Corollary~\ref{cor-sp>n=>reg},
\[
\text{
\ref{d-semi} \imp\ $x_0$ is irregular \imp\ $sp \le n$.
}
\]
From this it follows
that 
the condition $sp \le n$ 
 in 
\ref{d-notreg} and
\ref{d-semi-P-gen}
can equivalently be replaced by requiring that $x_0$ is irregular.
However, the following example shows that
the condition $sp \le n$
cannot be dropped from 
\ref{d-notreg} nor 
\ref{d-semi-P-gen}.

\begin{example} \label{ex-punct-ball}
We are now going to show that the condition $sp \le n$
cannot be dropped from 
\ref{d-notreg} nor
 \ref{d-semi-P-gen}
in Theorem~\ref{thm-semi-d}.
Assume that $sp > n$ and let $\Om=B(0,1) \setm \{0\}$.
Then $x_0:=0$ is regular, by Corollary~\ref{cor-sp>n=>reg}.

\ref{d-notreg}
Let  $g \in \Vsp(\Om)$. 
As integrals 
do not see sets of measure zero, we directly 
get that $g \in \Vsp(B(0,1))$.  
Since $sp>n$, it follows from the fractional Sobolev inequality
(see Di Nezza--Palatucci--Valdinoci~\cite[Theorem~8.2]{DNPV12})
that there is a 
function $\gt =g$ a.e.\ in $\R^n$ such that $\gt \in C(B(0,1))$.
Hence it follows from Theorem~\ref{thm-main-reg}\ref{a-contx0} that
\[
    \lim_{\Om \ni x \to 0} H g(x) 
    =\lim_{\Om \ni x \to 0} H \gt(x) = \gt(0),
\]
i.e.\ \eqref{eq-semireg-bdd} holds even though \ref{d-semi} fails.

\ref{d-semi-P-gen} 
Let $g : \Omc \to \R$ be a bounded measurable function.
Since $x_0$ is isolated in $\Omc$, $g$ is continuous at $x_0$,
and thus, by Theorem~\ref{thm-main-reg}\ref{a-contx0-P},
\[
    \lim_{\Om \ni x \to x_0} \uP g(x)=g(x_0),
\]  
i.e.\ \eqref{lim-Pg-ex} holds even though \ref{d-semi} fails.
\end{example}

In order to prove Theorem~\ref{thm-semi-d} we will need
the following lemma.

\begin{lem} \label{lem-ROm}
Assume that $sp \le n$.
Then the set
\[
R:=\{x \in \bdy \Om: x \text{ is regular}\}
\]
does not have any isolated points.
\end{lem}

The punctured ball in Example~\ref{ex-punct-ball}
shows that this is false for $sp >n$.

\begin{proof}
By Remark~10.5 in~\cite{BBK1}, 
$R$ is dense in $\bdy \clOm$, which
does not have any isolated points.
So consider $x_0 \in  R \setm \bdy \clOm$
and let $r>0$ be so small that
$B=B(x_0,r) \subset \clOm$.
Since $x_0$ is not semiregular and $\Csp(\{x_0\})=0$, by 
Lemma~\ref{lem-sp<=n}, we get that 
\[
   \Csp((B \setm \{x_0\}) \cap \bdy \Om)
   =\Csp(B  \cap \bdy \Om)
 >0,
\]
by 
Propositions~\ref{prop-S-largest} and~\ref{prop-Cpt-subadd}.
Hence, by the Kellogg property (Theorem~\ref{thm-kellogg})
there is a regular point in $(B \setm \{x_0\}) \cap \bdy \Om$.
Letting $r \to 0$ shows that $x_0$ is not isolated in~$R$.
\end{proof}

\begin{proof}[Proof of Theorem~\ref{thm-semi-d}]
\ref{d-semi}\imp\ref{d-notreg}
By Proposition~\ref{prop-S-largest}, there is a ball $B \ni x_0$ 
such that $\Csp(B \setm \Om)=0$.
In particular, $\Csp(\{x_0\})=0$, 
and hence $sp \le n$, by Lemma~\ref{lem-sp<=n}.

Since $Hg \in \Vsp(\Om)$ is $\LL$-harmonic in $\Om$
it follows from the removability
Theorem~\ref{thm-remove}
that $Hg$ is a solution 
of $\LL u=0$ in $\Om \cup B$.
Hence, there is an $\LL$-harmonic (and thus continuous)
function $u$ in $\Om \cup B$ such that $u=Hg$ a.e.
Since $Hg$ is continuous in $\Om$,
we have   
$u \equiv Hg$ in $\Om$ and thus
\[
    \lim_{\Om \ni x \to x_0} Hg(x)
  =   \lim_{\Om \ni x \to x_0} u(x)
  = u(x_0) 
\quad \text{exists and is finite}.
\]

\ref{d-notreg}\imp\ref{d-liminf}
First, we show that $x_0$ is irregular.
Assume for a contradiction that $x_0$ is regular.
Without loss of generality we can assume that $x_0=0$.
Since $sp \le n$, it follows from 
Lemma~\ref{lem-ROm} that there is a sequence 
$\{x_j\}_{j=1}^\infty$ of regular boundary points such that
$|x_{j+1}| \le \tfrac12 |x_j|$, $j=1,2,\dots$\,.
By Lemma~\ref{lem-sp<=n},
\[
   \csp(\{x_j\},B(x_j,\tfrac14 |x_j|))=0.
\] 
Hence it follows from the definition of the condenser capacity,
truncation and the 
fractional Poincar\'e inequality \eqref{eq-frac-PI}
that  there are $u_j \in \Lipc(B(x_j,\tfrac14 |x_j|))$ 
with $u_j(x_j)=1$, $0 \le u_j \le 1$ in $\R^n$ 
and $\|u_j\|_{\Wsp(\Rn)}< 2^{-j}$, $j=1,2,\dots$\,.

The functions $u_j$ have pairwise disjoint supports.
Hence 
\begin{equation}    \label{eq-def-g}
     g:= \sum_{j=1}^\infty (-1)^j u_j \in \Wsp(\Rn) = \Vsp(\Rn)
\end{equation}
is continuous at each $x_j$, $j=1,2,\dots$\,, and $|g| \le 1$.
Thus by 
Theorem~\ref{thm-main-reg}\ref{a-contx0},
\[
    \lim_{\Om \ni x \to x_j} H g(x) = g(x_j) = (-1)^j,
\quad j=1,2,\dots.
\]
Since $|g| \le 1$ it follows that
\[
 \liminf_{\Om \ni x \to x_0} H g(x) = -1 
\quad \text{and} \quad 
 \limsup_{\Om \ni x \to x_0} H g(x) =1,
\]
which contradicts \ref{d-notreg}.
Hence $x_0$ must be irregular.

By \ref{d-notreg}, the limit $a:=\lim_{\Om \ni x \to x_0} H d_{x_0}(x)$ 
exists, and it is clearly nonnegative.
If $a$ were zero, then $x_0$ would be regular by 
Theorem~\ref{thm-main-reg}\ref{a-d}, but this contradicts the 
already proved irregularity.
Hence $a>0$ and \ref{d-liminf} must hold.

$\neg$\ref{d-R} $\imp$ $\neg$\ref{d-liminf}
For each $j \ge 1$, the set
$B(x_0,1/j)\cap \bdy \Om$ contains a regular boundary point $x_j$.
(The points $x_j$ need not be distinct.)
Since $x_j$ is regular, we can find $y_j \in B(x_j,1/j) \cap \Om$ so that
\[
        \frac{1}{j} > |d_{x_0}(x_j)-H d_{x_0}(y_j)|.
\]
Then $y_j \to x_0$ and $H d_{x_0}(y_j) \to 0$, as $j \to \infty$,
which gives $\neg$\ref{d-liminf}.

\ref{d-R}\imp\ref{d-semi}
It follows from the proof of Theorem~\ref{thm-trich}
that we are not in Case~2 therein. 
Hence we must be in Case~1 and thus $x_0$ is semiregular.

\ref{d-semi}\imp\ref{d-semi-P-gen}
The proof of this implication is almost the same 
as the proof of \ref{d-semi}\imp\ref{d-notreg} above,
but appealing to the removability Theorem~\ref{thm-remove-bdd}
instead of Theorem~\ref{thm-remove}.

\ref{d-semi-P-gen}\imp\ref{d-semi-P}
Since every bounded  $g \in C(\Omc)$ is resolutive, by
Theorem~\ref{thm-Perron},
we see that \eqref{eq-d-semi-P} follows directly from 
\ref{d-semi-P-gen}.
It remains to show that $x_0$ is irregular,
which  can be shown as in the proof
of 
\ref{d-notreg}\imp\ref{d-liminf}
above,
but appealing to 
Theorem~\ref{thm-main-reg}\ref{a-contx0-P}
instead of Theorem~\ref{thm-main-reg}\ref{a-contx0}.
(In fact, this is easier as in this case there is no need to
show that the function $g$ in \eqref{eq-def-g}
belongs to $\Vsp(\Rn)$.)

\ref{d-semi-P}\imp\ref{d-liminf}
Since $Hd_{x_0}=Pd_{x_0}$, by Theorem~\ref{thm-Perron},
we see that the limit
\[
    a:=\lim_{\Om \ni x \to x_0} H d_{x_0}(x) 
    = \lim_{\Om \ni x \to x_0} P d_{x_0}(x) 
   \quad \text{exists}.
\]
Since $x_0$ is irregular, 
it follows from Theorem~\ref{thm-main-reg}\ref{a-d} that $a>0$.
\end{proof}

For the local equation $\Delta_p u=0$,
another characterization of semiregularity is that there is $r>0$ such that
the capacity of $B(x_0,r) \cap \bdy \Om$ is zero,
see  Bj\"orn~\cite[(2.1)]{ABclass}.  
Next we show that such a characterization holds if $sp >1$.
On the other hand, 
Example~\ref{ex-sp<=1} below shows that  this 
does not characterize semiregularity when $sp \le 1$.

\begin{prop} \label{prop-semi-sp>1}
Assume that $sp>1$.
Let $B$ be a ball such that $B \cap \Om \ne \emptyset$.
Then 
\begin{equation} \label{eq-sp>1}
\Csp(B \cap \bdy \Om)=0 \quad \text{if and only if} \quad
\Csp(B \setm \Om)=0.
\end{equation}

In particular,
$x_0 \in \bdy \Om$ is semiregular
if and only if
\begin{equation} \label{eq-sp>1-semi}
\Csp(B(x_0,r) \cap \bdy \Om)=0 
\quad \text{for some $r>0$}.
\end{equation}
\end{prop}

\begin{proof}
If $B \subset \clOm$, then the equivalence in~\eqref{eq-sp>1} is trivial.
Otherwise, 
$B \cap\bdy \Om$ splits $B$ into the two nonempty open sets
$B \cap \Om$ and $B \setm \clOm$.
It follows that the $(n-1)$-dimensional Hausdorff measure $H^{n-1}(B \cap \bdy \Om)>0$.
Since $sp>1$,
Theorem~5.1.13 in Adams--Hedberg~\cite{AH} (with $h(t)=t^{n-1}$), 
together with Remark~\ref{rmk-Bsp-Csp} below,
then implies that 
\[
\Csp(B \setm \Om) \ge \Csp(B \cap  \bdy \Om)>0,
\]
which shows the equivalence in~\eqref{eq-sp>1} also in this case.

The last part now follows directly from Proposition~\ref{prop-S-largest}.
\end{proof}

\begin{example} \label{ex-sp<=1}
Assume that  $sp \le 1$.
  Let $x_0=0$, $B=B(0,1) \subset\R^n$
and $\Om=\{(x_1,\dots,x_n) \in B : x_n >0\}$.
Then 
it follows from Theorem~5.1.9 in Adams--Hedberg~\cite{AH}
(together with Remark~\ref{rmk-Bsp-Csp})
that 
$\Csp(B \cap \bdy \Om)=\Csp(B \cap \R^{n-1})=0$.
On the other hand, 
there is an exterior cone at $0$,
so it is a regular boundary point by the
Wiener criterion (Theorem~\ref{thm-main-reg}\ref{a-Wiener}).
Thus \eqref{eq-sp>1-semi} does not
characterize semiregularity when $sp \le 1$.
\end{example}

\section{Strongly irregular points}
\label{sect-strong-irr}

Recall from the introduction that an irregular boundary point 
$x_0 \in \bdy \Om$ is strongly irregular if 
\begin{Enumerate}
\renewcommand{\theenumi}{\textup{(\Roman{enumi})}}%
\renewcommand{\labelenumi}{\theenumi}%
\stepcounter{enumi}
\item \label{ab-sec}
For each $g \in \VspOm \cap C(\R^n)$, there is
a sequence $\{y_j\}_{j=1}^\infty$  such that
\begin{equation*}
\Om \ni y_j \to x_0 \text{ and } H g(y_j) \to g(x_0),
\quad \text{as } j \to \infty.
\end{equation*}
\end{Enumerate}

Also strongly irregular points can be characterized in various ways,
in particular, using only the function $d_{x_0}$.
Note that $Hd_{x_0}=Pd_{x_0}$, by Theorem~\ref{thm-Perron},
and thus $Hd_{x_0}$ can be replaced by $Pd_{x_0}$ in condition~\ref{st-liminf} below.

\begin{thm}\label{thm-str-irr}
  Let $x_0 \in \bdy \Om$ and $d_{x_0}=\min\{1,|x-x_0|\}$.
Then the following are equivalent\/\textup{:}
\begin{enumerate}
\item \label{st-str}
  $x_0$ is strongly irregular,
\item \label{st-bdd}
  there is 
a function $g \in \VspOm$
that is continuous at $x_0$ and
such that the limit
\begin{equation*} 
    \lim_{\Om \ni x \to x_0} H g(x) 
    \quad \text{does not exist},
\end{equation*}
\item \label{st-bdd-P}
  there is a bounded function $g:\Omc\to\R$ that is continuous at $x_0$ and
such that the limit
\begin{equation*} 
    \lim_{\Om \ni x \to x_0} \uP g(x) 
    \quad \text{does not exist},
\end{equation*}

\item \label{st-liminf}
   \[
    \limsup_{\Om \ni x \to x_0} H d_{x_0}(x) 
 >   \liminf_{\Om \ni x \to x_0} H d_{x_0}(x),
  \]
\item \label{st-R}
$x_0 \in \itoverline{R} \setm R$,
where $R=\{x \in \bdy \Om : x \text{ is regular}\}$.
\end{enumerate}
\end{thm}

\begin{proof}[Proof of Theorem~\ref{thm-str-irr}]
$\neg$\ref{st-liminf}\imp$\neg$\ref{st-str}
By assumption, the limit 
\[
c:= \lim_{\Om \ni x \to x_0} H d_{x_0}(x) \ge 0
\] 
exists.
If $c>0$, then $x_0$ is semiregular, by Theorem~\ref{thm-semi-d}\ref{d-liminf},
while $c=0$ implies that $x_0$ is regular, by Theorem~\ref{thm-main-reg}\ref{a-d}.

\ref{st-liminf}\imp\ref{st-bdd} 
This is trivial.

\ref{st-bdd}\imp\ref{st-str}
It follows from Theorem~\ref{thm-main-reg}\ref{a-contx0} 
that $x_0$ is irregular,
and from 
Theorem~\ref{thm-semi-d}\ref{d-notreg} that $x_0$ is not
semiregular.
Hence $x_0$ is strongly irregular, by Theorem~\ref{thm-trich}.

\ref{st-liminf}\imp\ref{st-bdd-P} 
Since 
$Pd_{x_0}=Hd_{x_0}$, by Theorem~\ref{thm-Perron},
this is immediate.

\ref{st-bdd-P}\imp\ref{st-str}
It follows from Theorem~\ref{thm-main-reg}\ref{a-contx0-P} 
that $x_0$ is irregular, and from 
Theorem~\ref{thm-semi-d}\ref{d-semi-P-gen} that $x_0$ is not
semiregular.
Hence $x_0$ is strongly irregular, by Theorem~\ref{thm-trich}.

\ref{st-str}\imp\ref{st-R}
By Theorem~\ref{thm-semi-d}\ref{d-R} and the irregularity of $x_0$, we
see that $x_0 \in \itoverline{R} \setm R$.

\ref{st-R}\imp\ref{st-str}
By Theorem~\ref{thm-semi-d}\ref{d-R}, $x_0$ is not semiregular,
and by assumption \ref{st-R} it is irregular. 
Hence $x_0$ must be strongly irregular, by 
Theorem~\ref{thm-trich}.
\end{proof}

A priori,  the sequence $\{y_j\}_{j=1}^\infty$ 
in the definition~\ref{ab-sec} of strongly irregular points
is allowed to depend on $g$.
However, 
we 
next show that there is a universal sequence $\{y_j\}_{j=1}^\infty$ 
suitable for all bounded functions.
The following result gives the precise statement.

\begin{thm} \label{thm-univ-seq}
Assume that $x_0 \in \bdy \Om$ is strongly irregular.
Then there is a 
sequence
$\{y_j\}_{j=1}^\infty$ such that
$\Om \ni y_j \to x_0$, as $j \to \infty$,  and
\begin{align}
 \lim_{j \to \infty} H g(y_j) &=g(x_0)
 \quad \text{for every bounded 
$g \in \VspOm$ that is continuous at $x_0$,} 
\label{eq-univ-H}
\\
 \lim_{j \to \infty} \uP g(y_j) &=g(x_0)
 \quad \text{for every bounded $g:\Om^c\to\R$
that is 
continuous at $x_0$.}
\label{eq-univ-P}
\end{align}
\end{thm}

Observe that it follows from the proof that any sequence
$\{y_j\}_{j=1}^\infty$ suitable for $d_{x_0}$ is a universal sequence.

\begin{proof}
Since $x_0$ is strongly irregular there is sequence 
$\{y_j\}_{j=1}^\infty$ such that
$\Om \ni y_j \to x_0$, as $j \to \infty$,
and
\[
 \lim_{j \to \infty} H d_{x_0}(y_j) =0,
\quad \text{where }   d_{x_0}(x):=\min\{1,|x-x_0|\}.
\]
Now consider a bounded function
$g \in \VspOm$ that is 
continuous at $x_0$.
We may assume that $g(x_0)=0$.
Let $\eps>0$.
Since $g$ is continuous at $x_0$ and bounded,
there is $m>0$ such that $g \le md_{x_0}+\eps$ on $\Rn$.
Thus, by the comparison principle Theorem~\ref{thm-comparison},
\[
     \limsup_{j \to \infty} Hg(y_j) 
    \le  m    \limsup_{j \to \infty}  H d_{x_0}(y_j) + \eps
    = \eps.
\]
Letting $\eps \to 0$, shows that 
$     \limsup_{j \to \infty} Hg(y_j) \le 0$.
Applying this 
also to $-g$ shows that
\[
    0 \le \liminf_{j \to \infty} Hg(y_j) \le \limsup_{j \to \infty} Hg(y_j) \le 0,
\]
i.e.\ $\lim_{j \to \infty} Hg(y_j)=0$.
Thus \eqref{eq-univ-H} has been shown.

Next note that $Hd_{x_0}=Pd_{x_0}$, by Theorem~\ref{thm-Perron}.
The proof of \eqref{eq-univ-P} is now almost identical 
to the proof of \eqref{eq-univ-H}, the only difference is
that instead of appealing to the comparison principle Theorem~\ref{thm-comparison},
we use the fundamental fact
that 
\[
    \uP g_1 \le \uP g_2 \quad \text{if } g_1 \le g_2 \text{ on } \Rn,
\]
which is obvious from the definition of upper Perron solutions.
\end{proof}

The boundedness assumption in~\eqref{eq-univ-P} cannot
be dropped, see Remark~\ref{rmk-bdd-P-cannot-be-dropped}.
We do not know if it can be dropped in~\eqref{eq-univ-H}.
However, in the linear case, i.e.\ when $p=2$, 
the boundedness assumption in~\eqref{eq-univ-H} can be dropped.
This is the content of the following result.

\begin{thm}\label{thm-str-irr-linear}
Assume that $p=2$
and that $x_0 \in \bdy \Om$ is strongly irregular.
Then there is a  sequence
$\{y_j\}_{j=1}^\infty$ such that
$\Om \ni y_j \to x_0$, as $j \to \infty$,
and 
\begin{equation*}
 \lim_{j \to \infty} H g(y_j) =g(x_0)
 \quad \text{for every 
$g \in \VstOm$ that is  continuous at $x_0$.}
\end{equation*}
\end{thm}

For proving this 
we will need the following 
weak Harnack type inequality with a small tail term.

\begin{thm} \label{thm-KL-3.3}
\textup{(Kim--Lee~\cite[Theorem~3.3]{KL-orlicz})}
If $u \in V^{s, p}(\Omega)$ is a nonnegative solution of $\LL u=0$ in $\Omega$, 
then for any $x_0 \in \Rn$, $r>0$ and $0 < \delta <1$,
\begin{equation*}
\esssup_{B(x_0,r)} u_M \simle \delta  \Tail(u_M; x_0, r) + 
\delta^{-(p-1)n/sp^2}
\biggl(r^{-n} \int_{B(x_0,2r)} u_M^p \,dx \biggr)^{1/p},
\end{equation*}
where 
\begin{align*}
M & = \esssup_{B(x_0,2r) \setminus \Omega} u, 
\quad  u_M(x) = \max\{u(x),M \}, \\
\Tail(u; x_0, r) &= \biggl( r^{sp} \int_{\Rn \setminus B(x_0,r)} 
    \frac{|u(y)|^{p-1}}{|y-x_0|^{n+sp}} \,dy \biggr)^{1/(p-1)}
\end{align*}
and the comparison constant 
only depends on $n$, $s$, $p$ and $\Lambda$.
\end{thm}

\begin{proof}[Proof of Theorem~\ref{thm-str-irr-linear}]
Consider the sequence $\{y_j\}_{j=1}^\infty$ given by 
Theorem~\ref{thm-univ-seq}
and let 
$g \in \VstOm$ be a (not necessarily bounded) function that is 
continuous at $x_0$.

We may assume that $x_0=0$
and $g(x_0)=0$, 
and also that 
$g \ge 0$, by considering $g=g_\limplus-g_\limminus$ and 
using the comparison principle (Theorem~\ref{thm-comparison}).
Let $0 < \eps  <1$.
By dominated convergence, 
$\|(g-t)_\limplus\|_{V^{s, 2}(\Om)} \to 0$, as $t \to \infty$.
Hence, there is $t>0$ such that $\|g_t\|_{V^{s, 2}(\Om)} < \eps$,
where $g_t=(g-t)_\limplus$. 
Then
$g=\min\{g,t\} + g_t$.
Since $g$ is continuous at $x_0$,
there is a ball $B:=B(0, r)$ such that $\sup_{2B}g<t$.

Let $u = Hg_t$ and 
$\delta = \eps^{2s/n}< 1$.
As $u=g_t=0$ in $2B \setminus \Om$, Theorem~\ref{thm-KL-3.3} 
with $M=0$ shows that
\begin{align}
\sup_{B} u
&\simle \delta \Tail(u; 0, r) + 
\delta^{-n/4s}
\biggl( r^{-n}\int_{2B} u^2 \,dx \biggr)^{1/2} \nonumber \\
&\simle
\eps^{2s/n}
\Tail(u; 0, r) + 
\eps^{-1/2}r^{-n/2}
\|u\|_{L^2(\Om)}.
\label{eq-Tail-est}
\end{align}

Since $u$ is a solution of $\LL u=0$ in $\Om$, 
by testing the equation with 
$u-g_t \in V^{s, 2}_0(\Om)$ we see that
\begin{equation*}
0=\E(u, u-g_t)=\iint_{(\Om^c \times \Om^c)^c} 
(u(x)-u(y))((u-g_t)(x)-(u-g_t)(y)) k(x, y) \,dy \,dx.
\end{equation*}
It then follows from Cauchy--Schwarz's 
inequality that
\begin{align*}
[u]_{V^{s, 2}(\Om)}^2
&\leq \Lambda \iint_{(\Om^c \times \Om^c)^c} (u(x)-u(y))^2 k(x, y) \,dy \,dx \\
&\leq \Lambda \biggl( \iint_{(\Om^c \times \Om^c)^c} (u(x)-u(y))^2 
   k(x, y) \,dy \,dx \biggr)^{1/2} \\
&\quad \times \biggl( \iint_{(\Om^c \times \Om^c)^c} (g_t(x)-g_t(y))^2 
    k(x, y) \,dy \,dx \biggr)^{1/2} \\
&\leq 2\Lambda^2 [u]_{V^{s, 2}(\Om)} [g_t]_{V^{s, 2}(\Om)},
\end{align*}
which shows that $[u]_{V^{s, 2}(\Om)} \leq 2\Lambda^2 [g_t]_{V^{s, 2}(\Om)}$. 
The fractional Poincar\'e inequality~\eqref{eq-frac-PI}
applied to $u-g_t \in V^{s, 2}_0(\Om)$ then gives
\begin{align*}
\|u-g_t\|_{L^2(\Om)} 
&  \simle [u-g_t]_{W^{s, 2}(\R^n)} 
 \le 2^{1/2}[u-g_t]_{V^{s,2}(\Om)} \\ 
&  \simle [u]_{V^{s, 2}(\Om)} + [g_t]_{V^{s,2}(\Om)}
 \simle \|g_t\|_{V^{s,2}(\Om)} 
 < \eps.
\end{align*}
It then follows that
\begin{align}\label{eq-Tail1}
\begin{split}
\Tail(u; 0, r) 
&\le \Tail(g_t; 0, r) + \Tail(u-g_t; 0, r) \\
& \le  \Tail(g; 0, r) + r^{-n} \|u-g_t\|_{L^1(\Om)}  \\
&\le \Tail(g; 0, r) + 
r^{-n} |\Om|^{1/2}  \|u-g_t\|_{L^2(\Om)} \\
&\simle \Tail(g; 0, r) + r^{-n} |\Om|^{1/2}
\end{split}
\end{align}
and that
\begin{equation}\label{eq-Tail2}
\|u\|_{L^2(\Om)} \le \|u-g_t\|_{L^2(\Om)} + \|g_t\|_{L^2(\Om)} \simle \varepsilon.
\end{equation}
Inserting \eqref{eq-Tail1} and~\eqref{eq-Tail2} into~\eqref{eq-Tail-est} shows that
\begin{equation*}
\limsup_{j \to \infty} u(y_j) 
\leq \sup_B u
\simle \eps^{2s/n} \bigl( \Tail(g; 0, r) + r^{-n}|\Omega|^{1/2} \bigr)
+\eps^{1/2}r^{-n/2}.
\end{equation*}
Note that $\Tail(g; 0, r) < \infty$ since $g\in \Vst(\Om)\subset L^{1}_{2s}(\Rn)$.
Hence, upon letting $\eps \to 0$,
we conclude that
\begin{equation*}
\lim_{j \to \infty} Hg_t(y_j) = \lim_{j \to \infty} u(y_j) = 0.
\end{equation*}
On the other hand, by 
Theorem~\ref{thm-univ-seq},
\begin{equation*}
\lim_{j \to \infty} H\min\{g, t\}(y_j) = 0.
\end{equation*}
Since $p=2$, 
the operator $\LL$ is linear and thus $Hg=H\min\{g, t\}+Hg_t$,
which finishes the proof.
\end{proof}

\section{Semiregularity for different values of
  \texorpdfstring{$(s,p)$}{(s,p)}}
\label{sect-different-sp}

\emph{In this section we also include $s=1$.}

\medskip

We will now also include the local case 
$\Delta_pu=0$ with $s=1$ in our investigation,
which is relevant for the comparison results below.
In this case we define $\Cop$ to be the Sobolev capacity  
associated with the classical  
Sobolev space $\Wp(\Rn)$ and
defined e.g.\ in Mal\'y--Ziemer~\cite[Definition~2.1]{MZ}
(with $r=1$ and ${\cal O}=\R^n$, and
denoted 
$\mathbf{C}_p(\cdot)=\mathbf{\ga}_{p,1}(\cdot\,;\R^n)$ therein.
Strictly speaking the definitions 
in~\cite{MZ} are only for $p \le n$, but it is straightforward
to consider them for all $1 < p <\infty$.)

For  $s=1$ we use the definition of semiregular boundary points
from Bj\"orn~\cite[Section~1]{ABclass}.
The following characterization of such semiregular points,
corresponding to 
Proposition~\ref{prop-S-largest},
is essential for the comparisons below.

\begin{prop} \label{prop-semi-s=1}
\textup{(\cite[Theorem~3.3]{ABclass})}
Assume that $s=1$ and 
$x_0 \in \bdy \Om$.
  Then   $x_0$  is semiregular
  if and only if there is $r>0$ such that $\Cop(B(x_0,r) \setm \Om)=0$.
\end{prop}

The following theorem was proved by Adams--Meyers~\cite[Theorem~5.5]{AM73},
see also  Adams--Hedberg~\cite[Theorem~5.5.1]{AH}.
They used the Bessel potential capacity (which is denoted
by $C_{s,p}$ in \cite{AH}).
The counterexamples they construct are 
compact Cantor sets,
and thus the inclusion  \eqref{eq-cap-subsetneq-K} below holds
whenever \eqref{eq-cap-subsetneq} holds.

\begin{remark} \label{rmk-Bsp-Csp}
For $s=1$, the Bessel potential space coincides with 
the Sobolev space  $\Wp(\Rn)$ (with equivalent norms), and so the
capacities are comparable.
For $0<s<1$ the Bessel potential space \emph{does not} coincide with 
the Sobolev space  $\Wsp(\Rn)$.
However, by Triebel~\cite[Theorem p.~172 and Remark~4 pp.~189--190]{Triebel95}
and 
Adams--Hedberg~\cite[Corollary~2.6.8 and Proposition~4.4.4]{AH}
the Bessel potential capacity 
is nevertheless comparable to our Sobolev capacity $\Csp$.
Hence the results 
from~\cite{AH} and~\cite{AM73} apply here.
\end{remark}

\begin{thm}   \label{thm-Adams-Meyers}
\textup{(\cite[Theorem~5.5]{AM73} or~\cite[Theorem~5.5.1]{AH})}
Let $0 < s_j\le 1 < p_j$, with $s_jp_j \le n$, $j=1,2$.  
Then 
\begin{equation}   \label{eq-cap-subsetneq}
\{E\subset\R^n: C_{s_1,p_1}(E)=0\} \subsetneq \{E\subset\R^n: C_{s_2,p_2}(E)=0\}
\end{equation}
if and only if either
\[
s_2 p_2 = s_1 p_1   
\text{ and } p_1 < p_2,
\quad \text{or} \quad  s_2 p_2 < s_1 p_1.  
\]
Moreover, in this case also
\begin{equation}   \label{eq-cap-subsetneq-K}
\{K\subset\R^n \text{ compact}: C_{s_1,p_1}(K)=0\} \subsetneq 
\{K\subset\R^n \text{ compact}: C_{s_2,p_2}(K)=0\}.
\end{equation}
\end{thm}  

In particular, if $(s_1,p_1) \ne (s_2,p_2)$  and $s_jp_j \le n$, $j=1,2$,
then either~\eqref{eq-cap-subsetneq} holds or
\[
\{E\subset\R^n: C_{s_2,p_2}(E)=0\} \subsetneq \{E\subset\R^n: C_{s_1,p_1}(E)=0\}.
\]

\begin{cor} \label{cor-semireg-inclusion}
Consider $0 < s_j\le1 < p_j$, $j=1,2$.
Then the implication
\begin{equation}   \label{eq-imp-semireg}
\text{$x_0$ is semiregular for $(s_1,p_1)$}
\imp
\text{$x_0$ is semiregular for $(s_2,p_2)$}
\end{equation}
holds, for all bounded open sets $\Om$ with $x_0 \in \bdy \Om$,
if and only if
any of the following  mutually disjoint cases holds\/\textup{:}
\begin{enumerate}
\item \label{g-a}
  $s_1 p_1 > n$,
\item \label{g-b}
  $s_2 p_2 = s_1 p_1\le n$ and $p_1 \le p_2$,
\item \label{g-c}
  $s_2 p_2 < s_1 p_1 \le n$,
\end{enumerate}  
\end{cor}  

To be able to also cover  $s=1$ in the proof below, 
we need a few results from the literature
corresponding to 
Lemmas~\ref{lem-sp<=n}, \ref{lem-zero-cap},
the Kellogg property (Theorem~\ref{thm-kellogg})
and its Corollary~\ref{cor-sp>n=>reg}.
We list these facts here:
\begin{enumerate}
\renewcommand{\theenumi}{\textup{(\roman{enumi})}}%
\item
Nonempty open sets have positive capacity 
and
$C_{1,p}(\{x_0\})=0$ if and only if
$p\le n$,
see e.g.\ Mal\'y--Ziemer~\cite[Theorem~2.8]{MZ}
and Heinonen--Kilpel\"ainen--Martio~\cite[Example~2.22 
and Corollary~2.29]{HeKiMa}.

\item
The Kellogg property was shown by 
Kilpel\"ainen~\cite[Corollary~5.6]{Kilp89} when $s=1$ and $p \le n$.
In an equivalent form,
the Kellogg property also follows from earlier
results by Hedberg~\cite[Corollary~1]{Hedberg72} (for $p>2-1/n$)
and Hedberg--Wolff~\cite[Theorem~2]{HedbergWolff} (for $p \le 2-1/n$), together with 
Maz{\cprime}ya's sufficiency part of the Wiener criterion \cite[Theorem, p.~236]{mazya70}.
Classical papers for $s=1$ usually only consider $p \le n$, as
regularity is trivial for $p>n$.
The Kellogg property for the full range $1<p<\infty$ 
can be found in~\cite[Theorem~9.11]{HeKiMa}.
In particular, all points are regular when $p>n$.
\end{enumerate}

\begin{proof}[Proof of Corollary~\ref{cor-semireg-inclusion}]
If \ref{g-a} holds, then $x_0$ is regular 
and thus not semiregular for $(s_1,p_1)$.
Hence, the implication \eqref{eq-imp-semireg} is trivial in this case.
If on the other hand \ref{g-b} or \ref{g-c} holds,
then it follows from 
Propositions~\ref{prop-S-largest} and~\ref{prop-semi-s=1},
together with Theorem~\ref{thm-Adams-Meyers},
that \eqref{eq-imp-semireg} holds.

Conversely, assume that all \ref{g-a}--\ref{g-c} fail.
If $s_2p_2 \le n$, then it follows from Theorem~\ref{thm-Adams-Meyers}
that there is a compact set  $K$ with 
\begin{equation} \label{eq-K}
 C_{s_2,p_2}(K)>0= C_{s_1,p_1}(K).
\end{equation}
If $s_2p_2 > n$ then  
\eqref{eq-K} holds with $K=\{x_0\}$   
since $s_1p_1\le n$ (because \ref{g-a} fails).

Let $B$ be a ball containing $K$ and let $\Om =B \setm K$.
As nonempty open sets always have positive capacity,
we
see that $K$ has empty interior and thus $K \subset \bdy \Om$.
It then follows from 
Propositions~\ref{prop-S-largest} and~\ref{prop-semi-s=1}
that each
$x \in K $ is semiregular for $(s_1,p_1)$.
On the other hand, by the Kellogg property 
there is some $x \in K$ which is regular for $(s_2,p_2)$.
Hence \eqref{eq-imp-semireg} fails.
\end{proof}

%%%%%%%%%%%%%%%%%%%%%%%%%%%%%%%%%%%%%%%%%%%%%%%%%%%%%%%%%%%%%%

\end{document}